\documentclass [11pt,reqno]{amsart}
\usepackage {amsmath,amssymb,verbatim,geometry,color}
\usepackage[all]{xy}
\usepackage{mathrsfs}
\usepackage[backref,pagebackref,pdftex,hyperindex]{hyperref}
\usepackage[pdftex]{graphicx}
\def\XXint#1#2#3{{\setbox0=\hbox{$#1{#2#3}{\int}$ }
\vcenter{\hbox{$#2#3$ }}\kern-.6\wd0}}

\newcommand{\A}{\mathbb{A}}
\newcommand{\B}{\mathbb{B}}
\newcommand{\C}{\mathbb{C}}
\newcommand{\G}{\mathbb{G}}

\newcommand{\N}{\mathbb{N}}
\renewcommand{\P}{\mathbb{P}}
 \newcommand{\Q}{\mathbb{Q}}
 \newcommand{\R}{\mathbb{R}}
 \newcommand{\Z}{\mathbb{Z}}

\newcommand{\fa}{\mathfrak{a}}
\newcommand{\fb}{\mathfrak{b}}
\newcommand{\fm}{\mathfrak{m}}

\newcommand{\hf}{\widehat{\f}}

\renewcommand{\hom}{\mathrm{hom}}

\newcommand{\cE}{\mathcal{E}}

\newcommand{\cH}{\mathcal{H}}

\newcommand{\cL}{\mathcal{L}}

\newcommand{\cO}{\mathcal{O}}

\newcommand{\cZ}{\mathcal{Z}}
\newcommand{\cX}{\mathcal{X}}

\newcommand{\nef}{\mathrm{nef}}

\renewcommand{\a}{\alpha}
\renewcommand{\b}{\beta}
\renewcommand{\d}{\delta}
\newcommand{\e}{\varepsilon}
\newcommand{\f}{\varphi}
\newcommand{\unipar}{\varpi}
\newcommand{\g}{\gamma}
\newcommand{\la}{\lambda}
\newcommand{\om}{\omega}

\newcommand{\Ga}{\Gamma}

\newcommand{\p}{\psi}

\newcommand{\eg}{{\rm e.g.\ }} 
\newcommand{\ie}{{\rm i.e.\ }} 
\newcommand{\loccit}{\textit{loc.\,cit.}}

\newcommand{\inter}{\cdot\ldots\cdot}

\newcommand{\an}{\mathrm{an}}

\newcommand{\gf}{\mathrm{gf}}

\DeclareMathOperator{\en}{E}

\DeclareMathOperator{\env}{P}

\DeclareMathOperator{\Zun}{Z^1}
\DeclareMathOperator{\Zunb}{Z^1_{\mathrm{b}}}

\DeclareMathOperator{\Numb}{N^1_{\mathrm{b}}}
\DeclareMathOperator{\Carb}{Car_{\mathrm{b}}}
\DeclareMathOperator{\Carbp}{Car^+_{\mathrm{b}}}
\DeclareMathOperator{\Nefb}{Nef_{\mathrm{b}}}

\DeclareMathOperator{\Cz}{C^0}

\DeclareMathOperator{\Div}{Div}

\DeclareMathOperator{\codim}{codim}

\DeclareMathOperator{\MA}{MA}

\DeclareMathOperator{\Spec}{Spec}
\DeclareMathOperator{\Amp}{Amp}
\DeclareMathOperator{\supp}{supp}

\DeclareMathOperator{\Pic}{Pic}

\DeclareMathOperator{\id}{id}

\DeclareMathOperator{\ord}{ord}
\DeclareMathOperator{\Nef}{Nef}

\DeclareMathOperator{\Psef}{Psef}
\DeclareMathOperator{\Bg}{Big}

\DeclareMathOperator{\Mov}{Mov}
\DeclareMathOperator{\PSH}{PSH}
\DeclareMathOperator{\CPSH}{CPSH}

\DeclareMathOperator{\redu}{red}

\DeclareMathOperator{\PL}{PL}

\DeclareMathOperator{\RPL}{\R\mathrm{PL}}

\DeclareMathOperator{\Hnot}{H^0}

\DeclareMathOperator{\Nz}{N}

\DeclareMathOperator{\te}{T}

\newcommand{\Num}{\operatorname{N^1}}

\DeclareMathOperator{\Vect}{Vect}

\newcommand{\ddc}{\operatorname{dd^c}\!}

\newcommand{\supstar}{\mathrm{sup}^\star }

\newcommand{\Gm}{\mathbb{G}_{\mathrm{m}}}

\renewcommand{\div}{\mathrm{div}}

\newcommand{\psef}{\mathrm{psef}}

\newcommand{\triv}{\mathrm{triv}}

\newcommand{\lin}{\mathrm{lin}}

\newcommand{\hto}{\hookrightarrow}

\newcommand{\D}{\Delta}

\newcommand{\cro}[1]{[\![#1]\!]}
\newcommand{\lau}[1]{(\!(#1)\!)}
\newcommand{\simto}{\overset\sim\to}

\numberwithin{equation}{section}       

\newtheorem{prop} {Proposition} [section]
\newtheorem{thm}[prop] {Theorem} 
\newtheorem{defi}[prop] {Definition}
\newtheorem{lem}[prop] {Lemma}
\newtheorem{cor}[prop]{Corollary}
\newtheorem{prop-def}[prop]{Proposition-Definition}

\newtheorem*{thmA}{Theorem A} 
 
\newtheorem*{thmB}{Theorem B}

\newtheorem{exam}[prop]{Example}
\newtheorem{rmk}[prop]{Remark}
\newtheorem{qst}[prop]{Question}

\theoremstyle{remark}

\title{Non-Archimedean Green's functions and Zariski decompositions}
\date{\today}

\author{S{\'e}bastien Boucksom \and Mattias Jonsson}

\address{Sorbonne Universit\'e and Universit\'e Paris Cit\'e\\
CNRS\\
IMJ-PRG\\
F-75005 Paris\\
France}
\email{sebastien.boucksom@imj-prg.fr}

\address{Dept of Mathematics\\
  University of Michigan\\
  Ann Arbor, MI 48109-1043\\
  USA}
\email{mattiasj@umich.edu}

\begin{document}

\begin{abstract}
  We study the non-Archimedean Monge--Amp\`ere equation on a smooth variety over a discretely or trivially valued field. First, we give an example of a Green's function, associated to a divisorial valuation, which is not $\Q$-PL (\ie not a model function in the discretely valued case).
  Second, we produce an example of a function whose Monge--Amp\`ere measure is a finite atomic measure supported in a dual complex, but which is not invariant under the retraction associated to any snc model. This answers a question by Burgos Gil et al in the negative. Our examples  are based on geometric constructions by Cutkosky and Lesieutre, and arise via base change from Green's functions over a trivially valued field; this theory allows us to efficiently encode the Zariski decomposition of a pseudoeffective numerical class.
\end{abstract}

\dedicatory{To the memory of Jean-Pierre Demailly, with admiration}

\maketitle

\setcounter{tocdepth}{1}
%
%
%
%

\section*{Introduction}
In the seminal paper~\cite{Yau78}, Yau studied the Monge--Amp\`ere equation
\begin{equation}
  (\om+\ddc\f)^n=\mu\tag{MA}
\end{equation}
on a compact $n$-dimensional K\"ahler manifold $(X,\om)$, where $\mu$ is a smooth, strictly positive measure on $X$ of mass $\int\om^n$, and $\f$ a smooth function on
$X$ such that the $(1,1)$-form $\om+\ddc\f$ is positive. Yau proved that there exists a smooth solution $\f$, unique up to a constant. If $\om$ is rational class, say $\om=c_1(L)$ for an ample line bundle $L$, then $\f$ can be viewed as a positive metric on $L$, and $(\om+\ddc\f)^n$ its the curvature measure.

As observed by Kontsevich, Soibelman, and Tschinkel~\cite{KS06,KT}, when studying degenerating 1-parameter families of K\"ahler manifolds, it can be fruitful to use non-Archimedean geometry in the sense of Berkovich over the field $\C\lau{\unipar}$ of complex Laurent series. In this context, a 
Monge--Amp\`ere operator was introduced by Chambert--Loir~\cite{CL06}, and a version of~(MA) was solved by the authors and Favre~\cite{nama}; see below.  Uniqueness of solutions was proved earlier by Yuan and Zhang~\cite{YZ17}.

Now, the method in~\cite{nama} is variational in nature, inspired by~\cite{BBGZ} in the complex case. It has the advantage of being able to deal with more general measures $\mu$, but the drawback of providing less regularity information on the solution. In fact,~\cite{nama} only gives a continuous solution, and is thus closer in spirit to~\cite{Kol98} than to~\cite{Yau78}.

It is therefore interesting to ask whether we can say more about the regularity of $\f$ in~(MA), at least for special measures $\mu$. In the non-Archimedean setting, there are many possible regularity notions; to describe the one we are focusing on, we first need to make the non-Archimedean version of (MA) more precise, following~\cite{nama,siminag}.

Let $X$ be a smooth projective variety over $K=\C\lau{\unipar}$ of dimension $n$. Consider a simple normal crossing (snc) model $\cX$ of $X$,  over the valuation ring $K^\circ=\C\cro{\unipar}$. The dual complex $\Delta_\cX$ embeds in the Berkovich analytification $X^\an$, and there is a continuous retraction $p_{\cX}\colon X^\an\to\Delta_\cX$.

A semipositive closed (1,1)-form on $X^\an$ in the sense of~\loccit\ is represented by a nef relative numerical class $\om\in\Num(\cX/\Spec K^\circ)$ for some snc model $\cX$. We assume that the image $[\om]$ of $\om$ in $\Num(X)$ is ample. In this case, there is a natural space $\CPSH(\om)=\CPSH(X,\om)$ of continuous $\om$-plurisubharmonic (psh) functions, and a Monge--Amp\`ere operator taking a function $\f\in\CPSH(\om)$
to a positive Radon measure $\f\to(\om+\ddc\f)^n$ on $X^\an$ of mass $[\om]^n$; see also~\cite{CLD} for a local theory. When $[\om]$ is rational, so that $[\om]=c_1(L)$ for an ample ($\Q$-)line bundle $L$ on $X$, we can view any $\f\in\CPSH(\om)$ as a semipositive metric on $L^\an$, with curvature measure $(\om+\ddc\f)^n$.

 As in~\cite{nama}, let us normalize the Monge--Amp\`ere operator and write
 \[
   \MA_\om(\f):=\tfrac1{[\om]^n}(\om+\ddc\f)^n.
 \]
 The main result in~\cite{nama} is that if $\mu$ is a Radon probability measure on $X^\an$ supported in some dual complex, then there exists $\f\in\CPSH(\om)$, unique up to an additive real constant, such that $\MA_\om(\f)=\mu$. More precisely, this was proved assuming that $X$ is defined over an algebraic curve, an assumption that was later removed in~\cite{BGJKM20}. Here we want to study whether for special measures $\mu$, the solution is regular in some sense.
 
 We first consider the class of \emph{piecewise linear} (PL) functions. A function $\f\in\Cz(X^\an)$ is ($\Q$-)PL if it is associated to a vertical $\Q$-divisor on some snc model, and PL functions are also known as \emph{model functions}. The set $\PL(X)$ of PL functions is a dense $\Q$-linear subspace of $\Cz(X^\an)$, and it is closed under taking finite maxima and minima.

 If $\f\in\PL(X)\cap\CPSH(\om)$, then the measure $\mu=\MA_\om(\f)$ is a rational divisorial measure, \ie a rational convex combination of Dirac masses at divisorial valuations. For example, when $[\om]=c_1(L)$ is rational, then the space $\PL(X)\cap\CPSH(\om)$ can be identified with the space of semipositive  \emph{model metrics} on $L^\an$, represented by a nef model $\cL$ of $L$, and $\MA_\om(\f)$ can be computed in terms of intersection numbers of $\cL$.

 Assuming $\om$ rational, one may ask whether, conversely, the solution to $\MA_\om(\f)=\mu$, with $\mu$ a rational divisorial measure, is necessarily PL\@.
 Here we focus on the case when $\mu=\d_w$ is a Dirac measure, where $w\in X^\div$ is a divisorial valuation. In this case, it was proved in~\cite{nama} that the solution $\f_w\in\CPSH(\om)$ to the Monge--Amp\`ere equation
 \begin{equation}
   \MA_\om(\f_w)=\d_w,\quad \f_w(w)=0\tag{$\star$}
 \end{equation}
 is the \emph{Green's function} of $w$, given by $\f_w=\sup\{\p\in\CPSH(\om)\mid \p(w)\le 0\}$.
 \begin{thmA}
   Assume that $\om$ is a rational semipositive closed (1,1)-form with $[\om]$ ample, and that $w\in X^\div$ is a divisorial valuation.
   Let $\f_w\in\CPSH(\om)$ be the Green's function satisfying~($\star$) above. Then:
   \begin{itemize}
   \item[(i)]
     in dimension 1, $\f_w\in\PL(X)$;
   \item[(ii)]
     in dimension $\ge 2$, it may happen that $\f_w\not\in\PL(X)$.
   \end{itemize}
 \end{thmA}
 Writing $[\om]=c_1(L)$, Theorem~A says that the metric on $L^\an$ corresponding to $\f_w$ is a model metric in dimension 1, but not necessarily in dimension 2 and higher. This answers a question in~\cite{nama}, see Remark~8.8 in~\loccit
 
 Here~(i) is well known, for example from the work of Thuillier~\cite{thuillierthesis}; we give a proof in~\S\ref{sec:green}. As for~(ii), we present one example where $X$ is an abelian surface, and another one where $X=\P^3$. See Examples~\ref{exam:Green1} and~\ref{exam:Green3}.
 
We will discuss the structure of these examples shortly, but mention here that they are both \emph{$\R$-PL}, \ie they belong to the space the smallest $\R$-linear subspace $\RPL(X)$ of $\Cz(X^\an)$ containing $\PL(X)$ and stable under max and min.
The question then arises whether 
also in higher dimension, the solution $\f_w$ to~($\star$) is $\R$-PL for any divisorial valuation $v$. While we don't have a counterexample to this exact question (with $\omega$ rational, but see~Example~\ref{exam:Les1}), we prove that the situation can be quite complicated in dimension three and higher.

Namely, let us say that a function $\f\in\Cz(X^\an)$ is \emph{invariant under retraction} if $\f=\f\circ p_\cX$ for some (and hence any sufficiently high) snc model $\cX$. For example, a function in $X^\an$ is $\R$-PL iff it is invariant  under retraction and its restriction to any dual complex $\D_\cX$ is $\R$-PL in the sense that it is affine on the cells of some subdivision of $\D_\cX$ into real simplices.

If $\f\in\CPSH(\om)$ is invariant under retraction, say $\f=\f\circ p_\cX$, then the Monge--Amp\`ere measure $\MA_\om(\f)$ is supported in $\Delta_\cX$.
However, if $\mu$ is supported in $\Delta_\cX$, then the solution $\f$ to $\MA_\om(\f)=\mu$ may not satisfy $\f=\f\circ p_\cX$, see~\cite[Appendix~A]{GJKM}. Still, one may ask whether $\f$ is invariant under retraction, that is, $\f=\f\circ p_{\cX'}$ for any sufficiently high snc model $\cX'$, see Question~2 in~\loccit.  A version of this question (see Remark~\ref{rmk:compprop}) in the context of Calabi--Yau varieties plays a key role in the recent work of Yang Li~\cite{Li20}, see also~\cite{HJMM,Li23,AH23}. Our next result provides a negative answer in general.
\begin{thmB}
  Let $X=\P^3_K$, with $K=\C\lau{\unipar}$, and let $\om$ be the closed (1,1)-form associated to the numerical class of $\cO(1)$ on $\P^3_{K^\circ}$. Then there exists $\f\in\CPSH(\om)$ such that $\MA_\om(\f)$ has finite support in some dual complex, but $\f$ is not invariant under retraction. In particular, $\f\not\in\RPL(X)$.
\end{thmB}

Let us now say more about the examples underlying Theorem~B and Theorem~A~(ii). They all arise in the \emph{isotrivial case}, when the variety $X$ over $K$ is the base change of a smooth projective variety $Y$ over $\C$, and the $(1,1)$-form is defined by the pullback of an ample numerical class $\theta\in\Num(Y)$ to the trivial (snc) model $Y_{K^\circ}$ of $X=Y_K$. In this case, we can draw on the global pluripotential theory over a trivially valued field developed in~\cite{trivval}, a theory which interacts well with algebro-geometric notions such as diminished base loci and Zariski decompositions of pseudoeffective classes.

Specifically, given a smooth projective complex variety $Y$, and an ample numerical class $\theta\in\Num(Y)$, we have a convex set $\CPSH(\theta)=\CPSH(Y,\theta)\subset\Cz(Y^\an)$ of continuous $\theta$-psh functions, where $Y^\an$ now denotes the Berkovich analytification of $Y$ with respect to the \emph{trivial} absolute value on $\C$. A \emph{divisorial valuation} on $Y$ is of the form $v=t\ord_E$, where $t\in\Q_{\ge0}$ and $E\subset Y'$ is a prime divisor on a smooth projective variety $Y'$ with a proper birational morphism $Y'\to Y$. When instead $t\in\R_{\ge0}$,  we say that $v$ is a \emph{real divisorial valuation}.
If $\Sigma\subset Y^\an$ is a finite set of real divisorial valuations, then we consider the Green's function of $\Sigma$, defined as
  \[
    \f_\Sigma:=\sup\{\f\in\CPSH(Y,\theta)\mid \f|_\Sigma\le 0\}.
  \]
  By~\cite{trivval}, $\f_\Sigma\in\CPSH(Y,\theta)$, and the Monge--Amp\`ere measure of $\f_\Sigma$ is supported in $\Sigma$.

  The base change $X=Y_{\C\lau{\unipar}}\to Y$ induces a surjective map $\pi\colon X^\an\to Y^\an$, and this map admits a canonical section $\sigma\colon Y^\an\to X^\an$, called \emph{Gauss extension}, and whose image consists of all $\C^\times$-invariant points in $X^\an$. For any $\f\in\CPSH(Y,\theta)$ we have $\pi^\star\f\in\CPSH(X,\om)$, and
  \[
    \MA_\om(\pi^\star\f)=\sigma_\star\MA_\theta(\f).
  \]
  In particular, if $v\in Y^\div$, then $\pi^\star\f_{\{v\}}$ is the Green function for $w:=\sigma(v)\in X^\div$. As both $\pi^\star$ and $\sigma^\star$ preserve the classes of $\Q$-PL and $\R$-PL functions, we see that in order to prove Theorem~A~(ii), it suffices to find a surface $Y$ and $v\in Y^\div$, such that $\f_v:=\f_{\{v\}}$ is not $\Q$-PL.

  Further, to prove Theorem~B, it suffices to find a finite set $\Sigma$ of real divisorial valuations on $Y$ such that $\pi^\star\f_\Sigma$ fails to be invariant under retraction. Indeed, the Gauss extension map $\sigma$ takes real divisorial valuations to Abhyankar valuations, and these are exactly the ones that lie in a dual complex. We then use the following criterion. Define the \emph{center} of any function $\f\in\PSH(Y,\theta)$ by 
  \[
    Z_Y(\f):=c_Y\{\f<\sup\f\},
  \]
  where $c_Y\colon Y^\an\to Y$ is the center map, see~\S\ref{sec:centpsh}. We show that if  $\pi^\star\f$ is invariant under retraction, then $Z_Y(\f)\subset Y$ is a strict Zariski closed subset, see Corollary~\ref{cor:critinvretr}.
It therefore suffices to find a Green's function $\f_\Sigma$ whose center is Zariski dense.
  
\smallskip
Our analysis of the Green's functions $\f_\Sigma$ is based on a relation between $\theta$-psh functions and families of b-divisors. Namely, we can pick a strict birational morphism $\rho\colon Y'\to Y$, with $Y'$ smooth, prime divisors $E_i\subset Y'$, and $c_i\in\R_{>0}$ such that $\Sigma=\{c_i^{-1}\ord_{E_i}\}$. If we set $D:=\sum_ic_i^{-1}E_i$, then we can express $\f_\Sigma$ in terms of the \emph{$b$-divisorial Zariski decomposition} of the numerical class $\rho^\star\theta-\la[D]$, for $\la\in(-\infty,\la_\psef]$, where $\la_\psef\in\R$ is the largest $\la$ such that this class is pseudoeffective (psef), see Theorem~\ref{thm:Greenb}. The analysis of the Zariski decomposition of a psef class $\theta$ in terms of $\theta$-psh functions is of independent interest.

Let us first consider the case of dimension two. The Zariski decomposition of  $\rho^\star\theta-\la D$ then depends in an $\R$-PL way on $\lambda$, and this implies that the Green's function $\f_\Sigma$ is $\R$-PL\@. On the other hand, $\f_\Sigma$ need not be $\Q$-PL\@. In fact, we prove in Theorem~\ref{thm:Greensurf} that $\f_\Sigma$ is $\Q$-PL iff the quantity
\[
  \te(\Sigma):=\sup\f_\Sigma
\]
is a rational number.
To prove Theorem~A~(ii), it therefore suffices to find a divisorial valuation $v$ on a surface $Y$ such that $\te(v)$ is irrational, and such examples can be found with $Y$ an abelian surface, and $v=\ord_E$ for a prime divisor $E$ on $Y$.

\smallskip
Using a geometric construction by Cutkosky~\cite{Cut}, we also give an example of a divisorial valuation $v$ on $Y=\P^3$ such that $\f_v$ is $\R$-PL but not $\Q$-PL for $\theta=c_1(\cO(1))$, see Example~\ref{exam:Cut}. Being $\R$-PL, this example is invariant under retraction. As explained above, in order to prove Theorem~B, it suffices to find $\Sigma$ such that $c_Y(\f_\Sigma)$ is a Zariski dense subset of $Y$. Using the notation above, we show that the center contains the image on $Y$ of the diminished base locus of the pseudoeffective class $\rho^\star\theta-\la_\psef [D]$ on $Y'$. We can then use a construction of Lesieutre~\cite{Les}, who showed that if $Y=\P^3$, $\theta=c_1(\cO(1))$, and $\rho\colon Y'\to Y$ is the blowup at nine very general points, then there exists an effective $\R$-divisor $D$ on $Y'$ supported on the exceptional locus on $\rho$, such that the diminished base locus of $\rho^\star\theta-D$ is Zariski dense. If we write $D=\sum_{i=1}^9c_iE_i$, then we can take $\Sigma=\{c_i^{-1}\ord_{E_i}\}$.


 \bigskip
\subsection*{Structure of the paper}
The article is organized as follows. In~\S\ref{sec:prelim} we recall some facts from birational geometry and pluripotential theory over a trivially valued field. This is used in~\S\ref{sec:pshbdiv} to relate $\theta$-psh functions and suitable families of $b$-divisors, after which we study the center of a $\theta$-psh function in~\S\ref{sec:centpsh}. In~\S\ref{sec:extremal} we define the extremal function $V_\theta\in\PSH(\theta)$ associated to a psef class: by evaluating this function at divisorial valuations we recover the minimal vanishing order of $\theta$ along a valuation. The extremal function is also closely related to various notions of Zariski decomposition of a psef class, as explored in~\S\ref{sec:Zariski}. After all this, we are finally ready to study Green's functions in~\S\ref{sec:greenzar} and~\S\ref{sec:exgreen}. Finally, in~\S\ref{sec:nontriv} and~\S\ref{sec:iso} we turn to the discretely valued case and prove Theorems~A and~B.
%
%
\subsection*{Notation and conventions}
A \emph{variety} over a field $F$ is a geometrically integral $F$-scheme of finite type.
We use the abbreviations
\emph{usc} for `upper semicontinuous',
\emph{lsc} for `lower semicontinuous',
and \emph{iff} for `if and only if'. 
%
%
\subsection*{Acknowledgement}
  The authors would like to thank Jos\'e Burgos, Antoine Ducros, Gerard Freixas, Walter Gubler, John Lesieutre and Milan Perera for useful exchanges related to this work. This article is dedicated to the memory of Jean-Pierre Demailly, whose extraordinary contributions to complex analytic and algebraic geometry have had a tremendous influence on our own research.

  The second author was partially supported by NSF grants DMS-1900025 and DMS-2154380.

%
%
%

%
%
\section{Preliminaries}\label{sec:prelim}
%
%
Throughout the paper---except in~\S\ref{sec:nontriv}---$X$ denotes a smooth projective variety over an algebraically closed field $k$ of characteristic $0$.
%
%
\subsection{Positivity of numerical classes and base loci}\label{sec:pos} 
We denote by $\Num(X)$ the (finite dimensional) vector space of numerical equivalence classes $\theta=[D]$ of $\R$-divisors $D$ on $X$. It contains the following convex cones, corresponding to various positivity notions for numerical classes:
\begin{itemize}
\item the \emph{pseudoeffective cone} $\Psef(X)$, defined as the closed closed cone generated by all classes of effective divisors;
\item the \emph{big cone} $\Bg(X)$, the interior of $\Psef(X)$;
\item the \emph{nef cone} $\Nef(X)$, equal to the closed convex cone generated by all classes of basepoint free line bundles;
\item the \emph{ample cone} $\Amp(X)$, the interior of $\Nef(X)$; 
\item the \emph{movable cone} $\Mov(X)$, the closed convex cone generated by all classes of line bundles whose base locus have codimension $2$.
\end{itemize} 
These cones satisfy
$$
\Nef(X)\subset\Mov(X)\subset\Psef(X),
$$
where the first (resp.~second) inclusion is an equality when $\dim X\le 2$ (resp.~$\dim X\le 1$), but is in general strict for $\dim X>2$ (resp.~$\dim X>1$). We will make use of the following simple property:
\begin{lem}\label{lem:movrestr} If $\theta\in\Num(X)$ is movable, then $\theta|_E\in\Num(E)$ is pseudoeffective for any prime divisor $E\subset X$. 
\end{lem}
The \emph{asymptotic base locus} $\B(D)\subset X$ of a $\Q$-divisor $D$ is defined as the base locus of $\cO_X(mD)$ for any $m\in\Z_{>0}$ sufficiently divisible. The \emph{diminished} (or \emph{restricted}) \emph{base locus} and the \emph{augmented base locus} of an $\R$-divisor $D$ are respectively defined as
$$
\B_-(D):=\bigcup_A\B(D+A)\quad\text{and}\quad \B_+(D):=\bigcap_A\B(D-A),
$$
where $A$ ranges over all ample $\R$-divisors such that $D-A$ (resp.~$D+A$) is a $\Q$-divisor. Since ampleness is a numerical property, these loci only depend on the numerical class $\theta=[D]\in\Num(X)$, and will be denoted by 
$\B_-(\theta)\subset\B_+(\theta)$. 

The augmented base locus $\B_+(\theta)$ is Zariski closed, and satisfies
\begin{equation*}
  \theta\in\Bg(X)  \Leftrightarrow\B_+(\theta)\ne X
  \quad\text{and}\quad
  \theta\in\Amp(X) \Leftrightarrow\B_+(\theta)=\emptyset. 
\end{equation*}
The diminished base locus satisfies 
\begin{equation}\label{equ:Bminplus}
\B_-(\theta)=\bigcup_{\e\in\Q_{>0}}\B_+(\theta+\e\om)
\end{equation} 
for any $\om\in\Amp(X)$. It is thus an at most countable union of subvarieties, which is not Zariski closed in general, and can even be Zariski dense (see~\cite{Les}). We further have
\begin{align*}
\theta\in\Psef(X) & \Leftrightarrow \B_-(\theta)\ne X;\\
\theta\in\Nef(X) & \Leftrightarrow \B_-(\theta)=\emptyset;\\
\theta\in\Mov(X) &\Leftrightarrow \codim\B_-(\theta)\ge 2.
\end{align*}

%
\subsection{The Berkovich space}\label{sec:Berko}
We use~\cite[\S 1]{trivval} as a reference. The \emph{Berkovich space} $X^\an$ is defined as the Berkovich analytification of $X$ with respect to the trivial absolute value on $k$~\cite{BerkBook}. We view it as a compact (Hausdorff) topological space, whose points are \emph{semivaluations}, \ie valuations $v\colon k(Y)^\times\to\R$ for some subvariety $Y\subset X$. We denote by $v_{Y,\triv}\in X^\an$ the trivial valuation on $k(Y)$, and set $v_\triv=v_{X,\triv}$. These trivial semivaluations are precisely the fixed points of the scaling action $\R_{>0}\times X^\an\to X^\an$ given by $(t,v)\mapsto tv$. 

We denote $X^\div\subset X^\an$ the (dense) subset of \emph{divisorial valuations}, of the form $v=t\ord_E$ with $t\in\Q_{\ge 0}$ and $E$ a prime divisor on a birational model $\pi\colon Y\to X$ (the case $t=0$ corresponding to $v=v_\triv$, by convention). In the present work, where $\R$-divisors arise naturally, it will be convenient to allow $t$ to be real, in which case we will say that $v=t\ord_E$ is a \emph{real divisorial valuation}. We denote by 
$$
X^\div_\R=\R_{>0} X^\div
$$
the set of real divisorial valuations. It is contained in the space $X^\lin\subset X^\an$ of \emph{valuations of linear growth} (see~\cite{BKMS} and~\cite[\S 1.5]{trivval}).
%
%
\subsection{Rational and real piecewise linear functions}\label{sec:PL}
In~\cite{trivval}, various classes of $\Q$-PL functions on $X^\an$  were introduced, and the purpose of what follows is to discuss their $\R$-PL counterparts. First, any ideal $\fb\subset\cO_X$ defines a homogeneous function 
$$
\log|\fb|\colon X^\an\to[-\infty,0]
$$
such that $\log|\fb|(v):=-v(\fb)$ for $v\in X^\an$.

Second, any \emph{flag ideal} $\fa$, \ie a coherent fractional ideal sheaf on $X\times\A^1$ invariant under the $\G_m$-action on $\A^1$ and trivial on $X\times\G_m$, defines a continuous function
\[
  \f_\fa\colon X^\an\to\R
\]
given by
$\f_\fa(v)=-\sigma(v)(\fa)$, where $\sigma\colon X^\an\to (X\times\A^1)^\an$ is the \emph{Gauss extension}. Concretely, we can write $\fa=\sum_{\la\in\Z}\fa_\la\unipar^{-\lambda}$ for a decreasing sequence of ideals $\fa_\la\subset\cO_X$ such that $\fa_\la=\cO_X$ for $\la\ll0$ and $\fa_\la=0$ for $\la\gg0$, and then $\f_\fa=\max_\la(\log|\fa_\la|+\la)$.

We then denote by: 
\begin{itemize}
\item $\PL^+_\hom(X)$ the set of $\Q_+$-linear combinations of functions of the form $\log|\fb|$ with $\fb\subset\cO_X$ a nonzero ideal; 
\item $\PL^+(X)$ the set of functions  $\f\in\Cz(X^\an)$ of the form $\f=\max_i\{\p_i+\la_i\}$ for a finite family $\p_i\in\PL^+_\hom(X)$ and $\la_i\in\Q$; equivalently, functions of the form $\f=\frac1m\f_\fa$ for a flag ideal $\fa$ and $m\in\Z_{>0}$;
\item $\PL(X)$ the set of differences of functions in $\PL^+(X)$, called \emph{rational piecewise linear functions} (\emph{$\Q$-PL functions} for short).  
\end{itemize}
The sets $\PL^+_\hom(X)$ are stable under addition and max, while $\PL(X)$ is a $\Q$-vector space, stable under max, and is dense in $\Cz(X^\an)$. 
\smallskip

As in~\cite[\S3.1]{trivval}, we denote by $\PL(X)_\R$ the $\R$-vector space generated by $\PL(X)$. It is not stable under max anymore; to remedy this, we further introduce: 
\begin{itemize}
\item the set $\PL^+(X)_\R$ of $\R_+$-linear combinations of functions in $\PL^+(X)$; 
\item the set $\RPL^+(X)$ of finite maxima of functions in $\PL^+(X)_\R$; 
e\item the set $\RPL(X)$ of differences of functions in $\RPL^+(X)$; we call its elements  \emph{real piecewise linear functions} (\emph{$\R$-PL functions} for short). 
\end{itemize}
As one immediately sees, the sets $\PL^+(X)_\R$ and $\RPL^+(X)$ are convex cones in $\Cz(X^\an)$, and $\RPL(X)$ is thus an $\R$-vector space. Further, $\RPL^+(X)$, and hence $\RPL(X)$, are clearly stable under max. Thus $\RPL(X)$ is the smallest $\R$-linear subspace of $\Cz(X^\an)$ that is stable under max and contains $\PL(X)$. 

Finally, introduce the convex cone $\PL^+_\hom(X)_\R$ of $\R_+$-linear combinations of functions in $\PL^+_\hom(X)$ (this is again not stable under max anymore). We then have:  

\begin{lem}\label{lem:RPL} A function $\f\in\Cz(X^\an)$ lies in $\RPL^+(X)$ iff $\f=\max_i\{\p_i+\la_i\}$ for a finite family $\p_i\in\PL^+_\hom(X)_\R$ and $\la_i\in\R$.
\end{lem}
\begin{proof} Since any function in $\RPL^+(X)$ is a finite max of functions $\f\in\PL^+(X)_\R$, it suffices to show that $\f$ is of the desired form. Write $\f=\sum_{i=1}^r t_i \f_i$ with $t_i\in\R_{>0}$ and $\f_i\in\PL^+(X)$, \ie $\f_i=\max_j\{\p_{ij}+\la_{ij}\}$ with $\p_{ij}\in\PL^+_\hom(X)$ and $\la_{ij}\in\Q$. Then 
$$
\f=\max_{j_1,\dots,j_r}\sum_{i=1}^r t_i\left(\p_{i j_i}+\la_{ij_i}\right). 
$$
Since each $\sum_it_i\p_{i j_i}$ lies in $\PL^+_\hom(X)_\R$, this shows that $\f$ is of the desired form. 

Conversely, assume $\f=\max_i\{\p_i+\la_i\}$ for a finite family $\p_i\in\PL^+_\hom(X)_\R$ and $\la_i\in\R$. For each $i$, write $\p_i=\sum_j t_{ij}\p_{ij}$ with $\p_{ij}\in\PL^+_\hom(X)\le 0$. Pick $v\in X^\an$ and $i$ such that $\f(v)=\p_i(v)+\la_i$. Since $\f$ is bounded, we can find $c\in\Q$ such that $\p_{ij}(v)\ge c$ for all $j$. This shows that $\f=\max_i\f_i$ with $\f_i:=\sum_j t_{ij}\max\{\p_{ij},c\}+\la_i$. For all $i,j$, $\max\{\p_{ij},c\}$ lies in $\PL^+(X)$, thus $\f_i\in\PL^+(X)_\R$, and hence $\f\in\RPL^+(X)$. 
\end{proof}
%
%
\subsection{Homogeneous functions vs.~$b$-divisors}\label{sec:bdiv}
We use~\cite[\S 1]{BdFF} and~\cite[\S 6.4]{trivval} as references for what follows. Recall that 
\begin{itemize}
\item a \emph{(real) $b$-divisor over $X$} is a collection $B=(B_Y)$ of $\R$-divisors on all (smooth) birational models $Y\to X$, compatible under push-forward as cycles, \ie an element of the $\R$-vector space 
$$
\Zunb(X)_\R:=\varprojlim_Y \Zun(Y)_\R;
$$
\item
  a $b$-divisor $B=(B_Y)$ is \emph{effective} if $B_Y$ is effective for all $Y$; if $B,B'$ are $b$-divisors, then we write $B\le B'$ iff $B'-B$ is effective;
\item a $b$-divisor $B\in\Zunb(X)_\R$ is said to be \emph{$\R$-Cartier} if there exists a model $Y$, called a \emph{determination} of $B$, such that $B_{Y'}$ is the pullback of $B_Y$ for all higher birational models $Y'$; thus the space of $\R$-Cartier $b$-divisors can be identified with
$$
\varinjlim_Y \Zun(Y)_\R.
$$
\end{itemize}
\begin{exam} Any $\R$-divisor $D$ on a model $Y\to X$ determines an $\R$-Cartier $b$-divisor $\overline{D}\in\Carb(X)_\R$, obtained by pulling back $D$ to all higher models, and any $\R$-Cartier $b$-divisor is of this form.
\end{exam}

For any $B\in\Zunb(X)_\R$ and $v\in X^\div$, we define $v(B)\in\R$ as follows: pick a prime divisor $E$ on a birational model $Y\to X$ and $t\in\Q_{\ge 0}$ such that $v=t\ord_E$, and set 
$$
v(B):=t\ord_E(B_Y).
$$
This is independent of the choices made, and the function $\p_B\colon X^\div\to\R$ defined by 
$$
\p_B(v):=v(B)
$$ 
is homogeneous (with respect to the scaling action of $\Q_{>0}$).

\begin{defi}\label{defi:divtype} We say that a homogeneous function $\p\colon X^\div\to\R$ is of \emph{divisorial type} if $\p(\ord_E)=0$ for all but finitely many prime divisors $E\subset X$.
\end{defi}
The next result is straightforward: 

\begin{lem}\label{lem:Bdivtype} The map $B\mapsto\p_B$ sets up a vector space isomorphism between $\Zunb(X)_\R$ and the space of homogeneous functions of divisorial type on $X^\div$. Moreover, $B\in\Zunb(X)_\R$ is effective iff $\p_B\ge0$.
\end{lem}
We endow $\Zunb(X)_\R$ with the topology of pointwise convergence on $X^\div$. If $\Omega$ is a topological space, then a map $f\colon\Omega\to\Zunb(X)$ is thus continuous iff $v\circ f\colon\Omega\to\R$ is continuous for all $v\in X^\div$.
We will also say that  $f\colon\Omega\to\Zunb(X)$ is lsc (resp.\ usc) iff $v\circ f\colon\Omega\to\R$ is lsc (resp.\ usc) for all $v\in X^\div$.

If $\Omega$ is a convex subset of a real vector space, then we say that $f\colon\Omega\to\Zunb(X)_\R$ is convex if $v\circ f$ is
is convex for all $v\in X^\div$.
This amounts to $f((1-t)x_0+tx_1)\le(1-t)f(x_0)+tf(x_1)$ for $x_0,x_1\in\Omega$, $0\le t\le 1$.
We say that $f$ is concave if $-f$ is convex.

Finally, if $\Omega\subset\R$ is an interval, then $f\colon\Omega\to\Zunb(X)_\R$ is increasing (resp.\ decreasing) if $v\circ f$ is increasing (resp.\ decreasing) for each $v\in X^\div$.


Next we will generalize~\cite[Theorem~6.32]{trivval} to real coefficients.

\begin{defi} We denote by $\Carbp(X)_\R$ the convex cone of divisors $B\in\Carb(X)_\R$ that are antieffective and relatively semiample over $X$. 
\end{defi}

\begin{prop}\label{prop:CarPL} The map $B\mapsto\p_B$ induces an isomorphism between $\Carb(X)_\R$ and the $\R$-vector space generated by (the restrictions to $X^\div$ of) all functions $\log|\fb|$ with $\fb\subset\cO_X$ a nonzero ideal. This isomorphism restricts to a bijection
  $$
\Carbp(X)_\R\simto\PL^+_\hom(X)_\R.
$$
\end{prop}
\begin{proof} The first point is a consequence of~\cite[Theorem~6.32]{trivval}, which also yields 
$$
\Carb(X)_\Q\simto\PL^+_\hom(X).
$$
Since the right-hand side generates the convex cone $\PL^+_\hom(X)_\R$, it suffices to show that the convex cone of antieffective and relatively semiample divisors in $\Carb(X)_\R$ is generated by antieffective and semiample divisors in $\Carb(X)_\Q$. By definition of a relatively semiample $\R$-Cartier $b$-divisor, we have $B=\sum_i t_i B_i$ with $t_i>0$ and $B_i\in\Carb(X)_\Q$ relatively semiample. By the Negativity Lemma (see~\cite[Proposition~2.12]{BdFF}), $B'_i:=B_i-\overline{B_{i,X}}$ is antieffective, and still relatively semiample. Denoting by $B_X=-\sum_\a c_\a E_\a$ the irreducible decomposition of the antieffective $\R$-divisor $B_X$, we infer 
$$
B=\sum_i t_i B'_i+\sum_\a c_\a(-\overline{E_\a})
$$
where $-\overline{E_\a}\in\Carb(X)_\Q$ is antieffective and relatively semiample. The result follows.  
\end{proof}
%
%
\subsection{Numerical $b$-divisor classes}\label{sec:numbdiv}
The space of \emph{numerical $b$-divisor classes} is defined as 
$$
\Numb(X):=\varprojlim_Y\Num(Y),
$$
equipped with the inverse limit topology (each finite dimensional $\R$-vector space $\Num(Y)$ being endowed with its canonical topology). 

Any $b$-divisor defines a numerical $b$-divisor class. This yields a 
natural quotient map 
$$
\Zunb(X)_\R\to\Numb(X)\quad B\mapsto [B].
$$ 
One should be wary of the fact this map
is \emph{not} continuous with respect to the topology of pointwise convergence of $\Zunb(X)_\R$. However, we observe: 
\begin{lem}\label{lem:bnumcont} For any finite set $\cE$ of prime divisors on $X$,  the quotient map $B\mapsto [B]$ is continuous on the subspace $\Zunb(X)_{\R,\cE}$ of $b$-divisors $B$ such that $B_X$ is supported by $\cE$. 
\end{lem}
\begin{proof} For any model $Y\to X$, each $B_Y$ with $B\in\Zunb(X)_{\R,\cE}$ lives in the finite dimensional vector space generated by the strict transforms of the elements of $\cE$ and the $\pi$-exceptional  prime divisors.  Thus $B\mapsto[B_Y]\in\Num(Y)$ is continuous on $\Zunb(X)_{\R,\cE}$, and the result follows. 
\end{proof}

The set of numerical classes of $\R$-Cartier $b$-divisors can be identified with the direct limit
$$
\varinjlim_Y\Num(Y)\subset\Numb(X).
$$
In particular, any numerical class $\theta\in\Num(X)$ defines a numerical $b$-divisor class $\bar\theta=(\theta_Y)_Y\in\Numb(X)$, where $\theta_Y$ is the pullback of $\theta$.

\begin{defi}
  The cone of \emph{nef $b$-divisor classes} 
$$
\Nefb(X)\subset \Numb(X)
$$
is defined as the closed convex cone generated by all classes of nef $\R$-Cartier $b$-divisors. 
\end{defi}

The following characterization is essentially formal (see~\cite[Lemma~2.10]{BdFF}). 

\begin{lem}\label{lem:nefmov} A $b$-divisor $B\in\Zunb(X)_\R$ is nef iff $B_Y$ is movable for all birational models $Y\to X$. In other words, $\Nefb(X)=\varprojlim_Y\Mov(Y)$. 
\end{lem}

We finally record the following version of the Negativity Lemma (see~\cite[Proposition~2.12]{BdFF}). 

\begin{lem}\label{lem:neg} If $B\in\Zunb(X)_\R$ is nef, then $B\le\overline{B_Y}$ for any birational model $Y\to X$. 
\end{lem} 

%
%
\subsection{Plurisubharmonic functions}
We use \cite[\S4]{trivval} as a reference. Given a $\Q$-line bundle $L\in\Pic(X)_\Q$ and a numerical class $\theta\in\Num(X)$, we denote by
\begin{itemize}
\item $\cH^\gf(L)=\cH^\gf_\Q(L)$ the set of \emph{generically finite Fubini--Study} functions for $L$, \ie functions $\f\colon X^\an\to\R\cup\{-\infty\}$ of the form 
$$
\f=m^{-1}\max_i\{\log|s_i|+\la_i\};
$$
where $m\in\Z_{>0}$ is sufficiently divisible, $(s_i)$ is a finite set of nonzero sections of $mL$, and $\la_i\in\Q$;
\item $\cH_\hom(L)\subset\cH^\gf(L)$ the set of \emph{homogeneous Fubini--Study functions}, for which the $\la_i$ can be chosen to be $0$; 
\item $\PSH(\theta)$ the set of \emph{$\theta$-psh} functions $\f\colon X^\an\to\R\cup\{-\infty\}$, $\f\not\equiv-\infty$, obtained as limits of decreasing nets $(\f_i)$ of generically finite Fubini--Study functions $\f_i$ for $\Q$-line bundles $L_i$ such that $c_1(L_i)\to\theta$;
\item $\PSH_\hom(\theta)\subset\PSH(\theta)$ the subset of homogeneous $\theta$-psh functions, that is, functions $\f\in\PSH(\theta)$ such that $\f(tv)=t\f(v)$ for $v\in X^\an$ and $t\in\R_{>0}$.
\end{itemize}
All functions in $\PSH(\theta)$ are finite valued on the set $X^\div\subset X^\an$ of divisorial valuations, and we endow $\PSH(\theta)$ with the topology of pointwise convergence on $X^\div$. For all $\f,\p\in\PSH(\theta)$, we further have 
$$
\f\le\p\text{ on }X^\div\Longleftrightarrow\f\le\p\text{ on }X^\an.
$$
In particular, the topology of $\PSH(\theta)$ is Hausdorff. The set of $\theta$-psh functions is preserved by the action of $\R_{>0}$ given by $(t,\f)\mapsto t\cdot\f$, where $(t\cdot\f)(v):=t\f(t^{-1}v)$. 
\begin{lem}\label{lem:pshpos} For any $\theta\in\Num(X)$ we have:
\begin{itemize}
\item[(i)] $\PSH(\theta)\ne\emptyset\Rightarrow\theta\in\Psef(X)$;
\item[(ii)] $0\in\PSH(\theta)\Leftrightarrow\theta\in\Nef(X)$; 
\item[(iii)] $\theta\in\Bg(X)\Rightarrow\PSH(\theta)\ne\emptyset$. 
\end{itemize}
\end{lem}
As we shall see in Proposition~\ref{prop:extfun}, (i) is in fact an equivalence, rendering~(iii) redundant. 
\begin{proof} For (i) and (ii) see~\cite[(4.1),(4.3)]{trivval}. If $\theta$ is big, we find a big $\Q$-line bundle $L$ such that $\theta-c_1(L)$ is nef. Then $\PSH(\theta)\supset\PSH(L)\supset\cH^\gf(L)\ne\emptyset$, which proves (iii). 
\end{proof}
\begin{exam}\label{exam:pdiv} For any effective $\R$-divisor $D$, $\p_D:=\p_{\overline D}$ satisfies $-\p_D\in\PSH_\hom([D])$. 
\end{exam}

Our assumption that $X$ is smooth and $k$ is of characteristic zero implies that the \emph{envelope property} holds, see~\cite[Theorem~A]{trivadd} for any class $\theta\in\Num(X)$. This means that if $(\f_\a)_\a$ is any family in $\PSH(\theta)$ that is uniformly bounded above, and $\f:=\sup_\a\f_\a$, then the usc regularization $\f^\star$ is $\theta$-psh.

The envelope property has many favorable consequences, as discussed in~\cite[\S5]{trivval}. For example, for any birational model $\pi\colon Y\to X$ and any $\theta\in\Num(X)$ we have
\begin{equation}\label{equ:pshbir}
  \PSH(\pi^\star\theta)=\pi^\star\PSH(\theta);
\end{equation}
see~\cite[Lemma~5.13]{trivval}.
%
%
\subsection{The homogeneous decomposition of a psh function}\label{sec:homog}
We refer to~\cite[\S 6.3]{trivval} for details on what follows. Fix $\theta\in\Num(X)$. For any $\f\in\PSH(\theta)$ and $\la\le\sup\f$, setting
\begin{equation}\label{equ:hom1}
\hf^\la:=\inf_{t>0}\{t\cdot\f-t\la\}
\end{equation}
defines a homogeneous $\theta$-psh function $\hf^\la\in\PSH_\hom(\theta)$.  The family $(\hf^\la)_{\la\le\sup\f}$ is further concave, decreasing, and continuous for the topology of $\PSH_\hom(\theta)$ (\ie pointwise convergence on $X^\div$), and it gives rise to the \emph{homogeneous decomposition} 
\begin{equation}\label{equ:hom2}
\f=\sup_{\la\le\sup\f}\{\hf^\la+\la\}. 
\end{equation}
For $\la=\sup\f=\f(v_\triv)$, the function $\hf^{\max}:=\hf^{\sup\f}$ computes the directional derivatives of $\f$ at $v_\triv$, \ie
\begin{equation}\label{equ:hmax}
\hf^{\max}(v)=\lim_{t\to 0_+}\frac{\f(tv)-\f(v_\triv)}{t}
\end{equation}
for $v\in X^\an$. The limit exists as the function $t\mapsto\f(tv)$ on $(0,\infty)$ is convex and decreasing, see~\cite[Proposition~4.12]{trivval}. 

\begin{exam}\label{exam:PLhom} Assume $\f=\f_{\fa}$ for a flag ideal $\fa=\sum_{\la\in\Z}\fa_\la\unipar^{-\la}$ on $X\times\A^1$. Then $\hf^{\max}=\log|\fa_{\la_{\max}}|$ where $\la_{\max}:=\max\{\la\in\Z\mid \fa_\la\ne 0\}$ (see~Example~6.28 in~\cite{trivval}).
\end{exam}

%
%
\section{Psh functions and families of $b$-divisors}\label{sec:pshbdiv}
We work with a fixed numerical class $\theta\in\Num(X)$.
%
%
\subsection{Homogeneous psh functions and b-divisors}
Recall that a function $\p\in\PSH_{\hom}(\theta)$ is uniquely determined by its values on $X^\div$. We say that $\p$ is of divisorial type if its restriction to $X^\div$ is of divisorial type, that is, $\p(\ord_E)=0$ for all but finitely many prime divisors $E\subset X$.

Slightly generalizing~\cite[Theorem~6.40]{trivval}, we show: 
\begin{prop}\label{prop:homogb} The map $B\mapsto\p_B$ in~\S\ref{sec:bdiv} sets up a 1--1 correspondence between: 
\begin{itemize}
\item[(i)] the set of $b$-divisors $B\in\Zunb(X)_\R$ such that $B\le 0$ and $\overline{\theta}+[B]\in\Numb(X)$ is nef;
\item[(ii)] the set of $\theta$-psh homogeneous functions $\p\in\PSH_\hom(\theta)$ of divisorial type.
\end{itemize}
\end{prop}
 \begin{proof} Pick $B$ as in (i). On the one hand, $\p_{\overline{B_X}}\in\PSH_\hom(-B_X)$,  see Example~\ref{exam:pdiv}. On the other hand, since $\overline{\theta}+[B]=\overline{(\theta+[B_X])}+([B]-\overline{[B_X]})$ is nef, it follows from~\cite[Theorem~6.40]{trivval} that $\p_{B-\overline{B_X}}=\p_B-\p_{\overline{B_X}}$ lies in $\PSH_\hom(\theta+B_X)$. Thus 
$$
\p_B\in\PSH(\theta+B_X)+\PSH(-B_X)\subset\PSH(\theta). 
$$
Conversely, pick $\p$ as in (ii), so that $\p=\p_B$ with $0\ge B\in\Zunb(X)_\R$. By~\cite[Corollary~6.17]{trivval}, we can write $\p$ as the pointwise limit of a decreasing net $(\p_i)$ such that $\p_i\in\cH_\hom(L_i)$ with $L_i\in\Pic(X)_\Q$ and $\lim_i c_1(L_i)=\theta$. Then $\p_i=\p_{B_i}$ for a Cartier $b$-divisor $0\ge B_i\in\Carb(X)_\Q$ such that $\overline{L_i}+B_i$ is semiample (see~\cite[Lemma~6.34]{trivval}), and hence $\overline{c_1(L_i)}+[B_i]\in\Numb(X)$ is nef. Further, $B_i\searrow B$ in $\Zunb(X)_\R$, and hence $[B_i]\to [B]$ in $\Numb(X)$ (see Lemma~\ref{lem:bnumcont}). Since $\overline{c_1(L_i)}+[B_i]$ is nef for all $i$, we conclude, as desired, that $\overline{\theta}+[B]$ is nef. 
\end{proof}
%
%
\subsection{Rees valuations}
In order to formulate a version of Proposition~\ref{prop:homogb} for general $\theta$-psh functions, the following notion will be useful.
\begin{defi}\label{defi:ReesR} Given any effective $\R$-divisor $D$ on $X$ with irreducible decomposition $D=\sum_\a c_\a E_\a$ on $X$, we denote by $\Ga_D\subset X^\div_\R$ the set of \emph{Rees valuations} of $D$, defined as the real divisorial valuations $v_\a:=c_\a^{-1}\ord_{E_\a}$. 
\end{defi} 
Note that $v_\a(D)=1$ for all $\a$. We can now state a variant of~\cite[Theorem~6.21]{trivval}: 

\begin{prop}\label{prop:homdom} Pick $\p\in\PSH_\hom(\theta)$, and an effective $\R$-divisor $D$ on $X$. Then
$$
\max_{\Ga_D}\p\le-1\Longleftrightarrow\p+\p_D\in\PSH_\hom(\theta-D). 
$$
\end{prop}
Recall that $0\ge-\p_D\in\PSH_\hom([D])$. 

\begin{proof} If $\p+\p_D\in\PSH_\hom(\theta-D)$, then $\p\le-\p_D$, and hence $\max_\Ga\p\le-1$, since $\p_D\equiv 1$ on $\Ga_D$. Conversely, assume $\max_{\Ga_D}\p\le-1$. Consider first the case where $\theta=c_1(L)$ for a $\Q$-line bundle and $\p\in\cH_\hom(L)$. For any $m$ sufficiently divisible we thus have $\p=\tfrac 1m\max_i\log|s_i|$ for a finite set of nonzero section $s_i\in\Hnot(X,mL)$. Using the notation of Definition~\ref{defi:ReesR}, we get for all $i$ and all $\a$ 
$$
c_\a^{-1}\ord_{E_\a}(s_i)=-\log|s_i|(v_\a)\ge m,
$$
and hence $\ord_{E_\a}(s_i)\ge \lceil m c_\a\rceil$. This implies $s_i=s_i' s_{D_m}$ for some $s_i'\in\Hnot(X,m(L-D_m))$, where 
$$
D_m:=m^{-1}\lceil mD\rceil=\sum_\a m^{-1}\lceil mc_\a\rceil E_\a
$$
and $s_{D_m}\in\Hnot(X,D_m)$ is the canonical section. Since $\p_{D_m}=-\log|s_{D_m}|$, we infer
$$
\p+\p_{D_m}=\tfrac 1m\max_i\log|s'_i|\in\cH_\hom(L-D_m)\subset\PSH_\hom(L-D_m). 
$$
When $m\to\infty$, $\p_{D_m}$ decreases to $\p_D$, and $[D_m]\to [D]$ in $\Num(X)$, and we infer $\p+\p_{D}\in\PSH_\hom(L-D)$. 

In the general case, $\p$ can be written as the pointwise limit of a decreasing net $\p_i\in\cH_\hom(L_i)$, where $L_i\in\Pic(X)_\Q$ satisfies that $c_1(L_i)-\theta$ is nef and tends to $0$ (see~\cite[Corollary~6.17]{trivval}). Pick $t\in (0,1)$. For all $i$ large enough and all $\a$, we then have $c_\a^{-1}\p_i(\ord_{E_\a})\le -t$, and hence 
$$
\p_i+t\p_{D}\in\cH_\hom(L_i-t D)\subset\PSH_\hom(L_i-tD)
$$
by the previous step of the proof. Since $\p_i+t\p_{D}$ decreases to $\p+t\p_D$ and $L_i-tD\to\theta-tD$ in $\Num(X)$, we infer $\p+t\p_D\in\PSH_\hom(\theta-t D)$ (see~\cite[Theorem~4.5]{trivval}). Pick any $\om\in\Amp(X)$. Then $\p+t\p_D\in\PSH_\hom(\theta-D+\om)$ for all $t\in(0,1)$ close to $1$, so by the envelope property (see~\cite[Theorem~5.11]{trivval}), we get $\p+\p_D\in\PSH_\hom(\theta-D+\om)$. As this is true for all $\om\in\Amp(X)$, we conclude $\p+\p_D\in\PSH_\hom(\theta-D)$ (again see~\cite[Theorem~4.5]{trivval}).
\end{proof}

%
%
\subsection{Psh functions and families of b-divisors}\label{sec:pshb}
We now extend Proposition~\ref{prop:homogb} to general $\theta$-psh functions. We say that $\f\in\PSH(\theta)$ is of divisorial type if the homogeneous psh function $\hf^{\max}\in\PSH_\hom(\theta)$ is of divisorial type, see~\S\ref{sec:homog}. By~\eqref{equ:hmax}, this is equivalent to $\f(\ord_E)=\sup\f$ for all but finitely many prime divisors $E\subset X$.
\begin{thm}\label{thm:pshb} There is a 1--1 correspondence between: 
\begin{itemize}
\item[(i)] the set of $\theta$-psh functions $\f\in\PSH(\theta)$ of divisorial type; 
\item[(ii)] the set of continuous, concave, decreasing families $(B_\la)_{\la\le\la_{\max}}$ of $b$-divisors, for some $\la_{\max}\in\R$, such that $B_\la\le 0$ and $\overline{\theta}+[B_\la]\in\Numb(X)$ is nef for all $\la\le\la_{\max}$. 
\end{itemize}
The correspondence is given by 
\begin{equation}\label{equ:pshb}
\f=\sup_{\la\le\la_{\max}}\{\p_{B_\la}+\la\},\quad \p_{B_\la}=\hf^\la. 
\end{equation}
In particular, we have $\la_{\max}=\sup\f$ and $\hf^{\max}=B_{\la_{\max}}$.
\end{thm}
\begin{proof} Pick a family $(B_\la)_{\la\le\la_{\max}}$ as in (ii). By Proposition~\ref{prop:homogb}, setting $\p_\la:=\p_{B_\la}$ defines a continuous, concave and decreasing family $(\p_\la)_{\la\le\la_{\max}}$ in $\PSH_\hom(\theta)$. Since $\theta$ has the envelope property, the usc upper envelope $\f:=\supstar_{\la\le\la_{\max}}(\p_\la+\la)$ lies in $\PSH(\theta)$. On $X^\div$, $\f$ coincides with $\sup_{\la\le\la_{\max}}(\p_\la+\la)$ (see~\cite[Theorem~5.6]{trivval}). By Legendre duality, we further have $\p_\la=\hf^\la$ for $\la<\la_{\max}$ (see~\cite[Theorem~6.24]{trivval}, and hence also for $\la=\la_{\max}$, by continuity of both sides on $(-\infty,\la_{\max}]$. 

Conversely, pick $\f$ as in (i), so that $\hf^{\max}\in\PSH_\hom(\theta)$ is of divisorial type. For each $\la\le\sup\f$ we then have $0\ge\hf^\la\ge\hf^{\max}$, which shows that $\hf^\la\in\PSH_\hom(\theta)$ is also of divisorial type. By Proposition~\ref{prop:homogb}, we thus have $\hf^\la=\p_{B_\la}$ for a $b$-divisor $B_\la\le 0$ such that $\overline\theta+[B_\la]$ is nef, and the family $(B_\la)_{\la\le\sup\f}$ is concave, decreasing and continuous, since so is $(\hf^\la)_{\la\le\sup\f}$.
\end{proof}
\begin{rmk}\label{rmk:divtype}
  Not every $\theta$-psh function is of divisorial type. For example, assume $\dim X=1$, and pick a sequence $(p_j)$ of closed points on $X$, with corresponding ideals $\fm_j\subset\cO_X$, and a sequence $\e_j$ in $\R_{>0}$ such that $\sum_j \e_j\le\deg\theta$. Then $\f:=\sum_j \e_j\log|\fm_j|\in\PSH(\theta)$, and $-c_j=\f(\ord_{p_j})<\sup\f=0$ for all $j$ (see~\cite[Example~4.13]{trivval}). 
\end{rmk}
%
%
\section{The center of a $\theta$-psh function}\label{sec:centpsh}
In this section we introduce the notion of the center of a $\theta$-psh function. This is a subset of $X$ defined in terms of the locus on $X^{\an}$ where $\f$ is smaller than its maximum.
%
%
\subsection{The center map}
For any $v\in X^\an$, we denote by $c_X(v)\in X$ its center, and by 
$$
Z_X(v):=\overline{\{c_X(v)\}}\subset X
$$
the corresponding subvariety.
The center map $c_X\colon X^\an\to X$ is surjective and anticontinuous, \ie the preimage of a closed subset is open. In particular, any subvariety $Z\subset X$ is of the form $Z=Z_X(v)$ for some $v$; we can simply take $v=\ord_Z$.

More generally, for any subset $S\subset X^\an$ we set 
\begin{equation}\label{equ:ZS}
Z_X(S):=\bigcup_{v\in S} Z_X(v).
\end{equation}
This is smallest subset of $X$ that contains $c_X(S)$ and is closed under specialization. 
%
%
\subsection{The center of a $\theta$-psh function}
We can now introduce
\begin{defi}\label{defi:centpsh} We define the \emph{center} on $X$ of any $\theta$-psh function $\f\in\PSH(\theta)$ as 
\[
  Z_X(\f):=Z_X(\{\f<\sup\f\}).
\]
\end{defi}

\begin{exam}\label{exam:ideal} For any nonzero ideal $\fb\subset\cO_X$, the function   $\p=\log|\fb|$ is $\theta$-psh if $\theta$ is sufficiently ample, and then $Z_X(\f)=V(\fb)$. More generally, if $\f=\sum_i t_i\log|\fb_i|$ with $t_i\in\R_{>0}$ and $\fb_i\subset\cO_X$ a nonzero ideal, then $Z_X(\f)=\bigcup_i V(\fb_i)$.
\end{exam}

Recall that to any $\theta$-psh function $\f\in\PSH(\theta)$ we can associate a homogeneous $\theta$-psh function $\hf^{\max}\in\PSH_\hom(\theta)$, see~\S\ref{sec:homog}.
\begin{lem}\label{lem:centhom}
  For any $\f\in\PSH(\theta)$ we have $\{\f<\sup\f\}=\{\hf^{\max}<0\}$. As a consequence,  $Z_X(\f)=Z_X(\hf^{\max})$. Moreover, the following conditions are equivalent:
  \begin{itemize}
  \item[(i)]
    $\f$ is of divisorial type;
  \item[(ii)]
    $\hf^{\max}$ is of divisorial type;
  \item[(iii)]
     $Z_X(\f)=Z_X(\hf^{\max})$ contains at most finitely many prime divisors $E\subset X$.
  \end{itemize}
\end{lem}
\begin{proof}
  Pick any $v\in X^\an$. By~\eqref{equ:hmax} and the fact that $t\mapsto\f(tv)$ is decreasing and convex, it follows that $\f(v)<\sup\f$ iff $\hf^{\max}(v)<0$. Thus $Z_X(\f)=Z_X(\hf^{\max})$ since $\sup\hf^{\max}=0$.

  Now the equivalence (i)$\Leftrightarrow$(ii) is definitional, and~(ii)$\Leftrightarrow$(iii) is clear since a prime divisor $E\subset X$ belongs to $Z_X(\hf^{\max})$ iff $\hf^{\max}(\ord_E)<0$. 
\end{proof}
Together with Example~\ref{exam:PLhom}, Lemma~\ref{lem:centhom} implies
\begin{cor}\label{cor:centflag}
  If $\f=\f_{\fa}$ for a flag ideal $\fa=\sum_{\la\in\Z}\fa_\la\unipar^{-\la}$ on $X\times\A^1$, then $Z_X(\f_\fa)=V(\fa_{\la_{\max}})$, where $\la_{\max}:=\max\{\la\in\Z\mid \fa_\la\ne 0\}$.
  \end{cor}

\begin{thm}\label{thm:centpsh1}
  For any $\f\in\PSH(\theta)$, the center $Z_X(\f)$ is a strict subset of $X$, and an at most countable union of (strict) subvarieties. Moreover, we have 
  $c_X^{-1}(Z_X(\f))=\{\f<\sup\f\}$. 
\end{thm}
\begin{proof}
  Note that $Z_X(\f)$ does not contain the generic point of $X$, so $Z_X(\f)\ne X$. Also note that by Lemma~\ref{lem:centhom} we may assume that $\f$ is homogeneous.

  If $\f\in\cH_\hom(L)$ for a $\Q$-line bundle $L$, so that $\f=\tfrac 1m\max_i\log|s_i|$ for a finite set of nonzero sections $s_i\in\Hnot(X,mL)$, then $Z_X(\f)=\bigcap_i(s_i=0)$, which is Zariski closed. In general, $\f$ can be written as the limit of a decreasing sequence $\f_m\in\cH_\hom(L_m)$ with $L_m\in\Pic(X)_\Q$ such that $c_1(L_m)\to\theta$ (see~\cite[Remark~6.18]{trivval}). For any $v\in X^\div$ we then have 
  \[
  c_X(v)\in Z_X(\f)\Leftrightarrow\f(v)<0\Leftrightarrow\f_m(v)<0\ \text{for some $m$},
\]
\ie $Z_X(\f)=\bigcup_m Z_X(\f_m)$, an at most countable union of strict subvarieties.

Next pick $v\in X^\an$, and set $Z=Z_X(v)$. By~\cite[Proposition~4.12]{trivval}, $\f(tv)=t\f(v)$ converges to $\f(v_{Z,\triv})=\sup_{Z^\an}\f$ as $t\to+\infty$, and hence $\f(v)<0\Leftrightarrow\f\equiv-\infty$ on $Z^\an$. By definition of the center, if $c_X(v)$ lies in $Z_X(\f)$, then we can find $w\in X^\an$ such that $\f(w)<0$ and $c_X(v)\in Z_X(w)$, \ie $Z\subset Z_X(w)$. Then $\f\equiv-\infty$ on $Z_X(w)\supset Z$, which yields $\f(v)<0$. Conversely, assume $\f(v)<0$, and hence $\f\equiv-\infty$ on $Z^\an$. We can find $w\in X^\div$ such that $Z=Z_X(w)$. Since $\f\equiv-\infty$ on $Z^\an=Z_X(w)^\an$,  we get $\f(w)<0$, and hence $c_X(v)\in Z_X(w)\subset Z_X(\f)$. 
\end{proof}  
For later use we record
\begin{lem}\label{lem:centsum}
  If $\f_i\in\PSH(\theta_i)$, $i=1,2$, then $Z_X(\f_1+\f_2)=Z_X(\f_1)\cup Z_X(\f_2)$.
\end{lem}
%
%
\subsection{Centers of PL functions}
%
%
The following result will play a crucial role in what follows.
\begin{lem}\label{lem:PLcent} If $\f\in\PSH(\theta)$ lies in $\RPL^+(X)$ (resp.~$\RPL(X)$), then $Z_X(\f)$ is Zariski closed (resp.~not Zariski dense) in $X$. 
\end{lem}

\begin{proof} Assume first $\f\in\RPL^+(X)$, and write $\f=\max_i\{\p_i+\la_i\}$ for a finite set $\p_i\in\PL^+_\hom(X)_\R$ and $\la_i\in\R$. As in Example~\ref{exam:PLhom}, we then have $\max_i\la_i=\sup\f$, and $\hf^{\max}=\max_{\la_i=\sup\f}\p_i$. This shows that 
$$
Z_X(\f)=Z_X(\hf^{\max})=\bigcap_{\la_i=\sup\f} Z_X(\p_i)
$$ 
is Zariski closed (see Example~\ref{exam:ideal}). Assume next $\f\in\RPL(X)$ and write $\f=\f_1-\f_2$ with $\f_1,\f_2\in\RPL^+(X)$. After replacing $\theta$ with a sufficiently ample class, we may assume that $\f_1,\f_2$ are $\theta$-psh. By~\eqref{equ:hmax} we have $\hf^{\max}=\hf_1^{\max}-\hf_2^{\max}$, and hence 
$$
Z_X(\f)=Z_X(\hf^{\max})\subset Z_X(\hf^{\max}_1)\cup Z_X(\hf^{\max}_2)=Z_X(\f_1)\cup Z_X(\f_2),
$$
which cannot be Zariski dense, since $Z_X(\f_1)$ and $Z_X(\f_2)$ are both Zariski closed strict subsets by the first part of the proof. 
\end{proof}
%
%
\section{Extremal functions and minimal vanishing orders}\label{sec:extremal}
Next we define a trivially valued analogue of an important construction in the complex analytic case.
%
%
\subsection{Extremal functions}
For any $\theta\in\Num(X)$, we define the \emph{extremal function} $V_\theta\colon X^\an\to [-\infty,0]$ as the pointwise envelope
\begin{equation}\label{equ:Vtheta}
V_\theta:=\sup\left\{\f\in\PSH(\theta)\mid\f\le 0\right\}.
\end{equation}

\begin{prop}\label{prop:extfun} For any $\theta\in\Num(X)$ we have 
\begin{align*}
\theta\in\Psef(X) & \Rightarrow V_\theta\in\PSH_\hom(\theta);\\
\theta\notin\Psef(X) & \Rightarrow V_\theta\equiv-\infty;\\
\theta\in\Nef(X) & \Leftrightarrow V_\theta\equiv 0.
\end{align*}
In particular, $\PSH(\theta)$ is nonempty iff $\theta$ is pseudoeffective. For any $\om\in\Amp(X)$, we further have
\begin{equation}\label{equ:Vdecr}
V_{\theta+\e\om}\searrow V_\theta\text{ as }\e\searrow 0. 
\end{equation}
\end{prop}
\begin{proof} Since the action $(t,\f)\mapsto t\cdot\f$ of $\R_{>0}$ preserves the set of candidate functions $\f$ in~\eqref{equ:Vtheta}, $V_\theta$ is necessarily fixed by the action, and hence homogeneous. If $\theta$ is not psef, then $\PSH(\theta)$ is empty (see Lemma~\ref{lem:pshpos}), and hence $V_\theta\equiv-\infty$. By Lemma~\ref{lem:pshpos}, we also have $V_\theta\equiv 0$ iff $\theta$ is nef. 

Next, assume $\theta\in\Bg(X)$. Then $\PSH(\theta)$ is non-empty (see Lemma~\ref{lem:pshpos}), and the envelope property implies that $V_\theta^\star$ is $\theta$-psh and nonpositive. It is thus a candidate in~\eqref{equ:Vtheta}, and hence $V_\theta^\star\le V_\theta$, which shows that $V_\theta^\star=V_\theta$ is $\theta$-psh. 

Assume now $\theta\in\Psef(X)$, and pick $\om\in\Amp(X)$. For each $\e>0$, the previous step yields $V_\e:=V_{\theta+\e\om}\in\PSH_\hom(\theta+\e\om)$. For $0<\e<\d$ we further have $\PSH(\theta)\subset\PSH(\theta+\e\om)\subset\PSH(\theta+\d\om)$, and hence $V_\d\ge V_\e\ge V_\theta$. Set $V:=\lim_\e V_\e$. For any $\d>0$ fixed, we have $V_\e\in\PSH_\hom(\theta+\d\om)$ for $\e\le\d$, and $V_\e\searrow V$ as $\e\to 0$. Thus $V\in\PSH_\hom(\theta+\d\om)$ for all $\d>0$, and hence $V\in\PSH_\hom(\theta)$. Since $V$ is a candidate in~\eqref{equ:Vtheta}, we get $V\le V_\theta$, and hence $V_\theta=V=\lim_\e V_\e$. This proves that $V_\theta$ is $\theta$-psh, as well as~\eqref{equ:Vdecr}. 
\end{proof}
%
%
\subsection{Minimal vanishing orders}
For $\theta\in\Psef(X)$, the function $V_\theta\in\PSH_\hom(\theta)$ is uniquely determined by its restriction to $X^\div$, where it is furthermore finite valued. For any $v\in X^\div$ we set 
\begin{equation}\label{equ:vV}
v(\theta):=-V_\theta(v)=\inf\{-\f(v)\mid\f\in\PSH(\theta),\,\f\le 0\}\in\R_{\ge 0}. 
\end{equation}
Note that 
\begin{equation}\label{equ:vpsef}
v(\theta)=\sup_{\e>0} v(\theta+\e\om)
\end{equation}
for any $\om\in\Amp(X)$, by~\eqref{equ:Vdecr}. As we next show, these invariants coincide with the minimal/asymptotic vanishing orders studied in~\cite{Nak,Bou,ELMNP}.  

\begin{prop}\label{prop:Vmin} Pick $v\in X^\div$. Then:  
\begin{itemize}
\item[(i)] the function $\theta\mapsto v(\theta)$ is homogeneous, convex and lsc on $\Psef(X)$; in particular, it is continuous on $\Bg(X)$; 
\item[(ii)] for any $\theta\in\Psef(X)$ we have 
\begin{equation}\label{equ:vnum}
v(\theta)\le\inf\left\{v(D)\mid D\equiv\theta\text{ effective }\R\text{-divisor}\right\},
\end{equation}
and equality holds when $\theta$ is big. 
\end{itemize}
\end{prop}
Note that equality in~\eqref{equ:vnum} fails in general for $\theta$ is not big, as there might not even exist any effective $\R$-divisor $D$ in the class of $\theta$. 

\begin{proof} Using $\PSH(\theta)+\PSH(\theta')\subset\PSH(\theta+\theta')$ and $\PSH(t\theta)=t\PSH(\theta)$ for $\theta,\theta'\in\Psef(X)$ and $t>0$, it is straightforward to see that $\theta\mapsto v(\theta)$ is convex and homogeneous on $\Psef(X)$. Being also finite valued, it is automatically continuous on the interior $\Bg(X)$. For any $\om\in\Amp(X)$ and $\e>0$, $\theta\mapsto v(\theta+\e\om)$ is thus continuous on $\Psef(X)$, and~\eqref{equ:vpsef} thus shows that $\theta\mapsto v(\theta)$ is lsc, which proves (i). 

Next pick $\theta\in\Psef(X)$. For each effective $\R$-divisor $D\equiv\theta$, the function $-\p_D\in\PSH_\hom(\theta)$, see Example~\ref{exam:pdiv}, is a competitor in~\eqref{equ:Vtheta}. Thus $-v(D)=\p_D(v)\le V_\theta(v)=-v(\theta)$, which proves the first half of (ii). Now assume $\theta$ is big, and denote by $v'(\theta)$ the right-hand side of~\eqref{equ:vnum}. Both $v(\theta)$ and $v'(\theta)$ are (finite valued) convex function of $\theta\in\Bg(X)$. They are therefore continuous, and it is thus enough to prove the equality $v(\theta)=v'(\theta)$ when $\theta=c_1(L)$ with $L\in\Pic(X)_\Q$ big. To this end, pick an ample $\Q$-line bundle $A$, and set $\om:=c_1(A)$. By~\cite[Theorem~4.15]{trivval}, for any $\e>0$ we can find $\f\in\cH^\gf(L+A)$ such that $\f\ge V_\theta$ and $\f(v_\triv)=\sup\f\le\e$. By definition, we have $\f=m^{-1}\max_i\{\log|s_i|+\la_i\}$ with $m$ sufficiently divisible and a finite family of nonzero sections $s_i\in\Hnot(X,m(L+A))$ and constants $\la_i\in\Q$. Then $\max_i\la_i=m\sup\f\le m\e$, and $m^{-1}v(s_i)=v(D_i)$ with $D_i:=m^{-1}\div(s_i)\equiv \theta+\om$, and hence $m^{-1}v(s_i)\ge v'(\theta+\om)$. Thus 
$$
-v(\theta)=V_\theta(v)\le\f(v)=m^{-1}\max_i\{v(s_i)+\la_i\}\le-v'(\theta+\om)+\e.
$$
This shows $v'(\theta)\ge v(\theta)\ge v'(\theta+\om)$, and hence $v'(\theta)=v(\theta)$, since $\lim_{\om\to 0} v'(\theta+\om)=v'(\theta)$ by continuity on the big cone. 
\end{proof}

\begin{rmk}\label{rmk:asymvan} If $L\in\Pic(X)$ is big, then \cite[Corollary~2.7]{ELMNP} (or, alternatively, a small variant of the above argument) shows that $v(c_1(L))$ is also equal to the asymptotic vanishing order 
\begin{align*}
v(\|L\|) : & =\lim_{m\to\infty}\tfrac{1}{m}\min\left\{v(s)\mid s\in\Hnot(X,mL)\setminus\{0\}\right\}\\
& =\inf\left\{v(D)\mid D\sim_\Q L\text{ effective }\Q{-divisor}\right\}. 
\end{align*}
\end{rmk}

\begin{rmk}\label{rmk:mincont} Continuity of minimal vanishing orders beyond the big cone is a subtle issue. For any $v\in X^\div$, the function $\theta\mapsto v(\theta)$, being convex and lsc on $\Psef(X)$, is automatically continuous on any polyhedral subcone (cf.~\cite{How}). When $\dim X=2$, it is in fact continuous on the whole of $\Psef(X)$, but this fails in general when $\dim X\ge 3$ (see respectively Proposition~III.1.19 and Example~IV.2.8 in \cite{Nak}). 
\end{rmk}

%
%
\subsection{The center of an extremal function}
The following fact plays a key role in what follows. 
\begin{thm}\label{thm:Vdivtype} For any $\theta\in\Psef(X)$, the function $V_\theta\in\PSH_\hom(\theta)$ is of divisorial type (see Definition~\ref{defi:divtype}). Further, its center $Z_X(V_\theta)$ coincides with the diminished base locus $\B_-(\theta)$ (see~\S\ref{sec:pos}). 
\end{thm}
The proof relies on the next result, which corresponds to~\cite[Corollary~III.1.11]{Nak} (see also~\cite[Theorem~3.12]{Bou} in the analytic context). 

\begin{lem}\label{lem:finitediv} Pick $\theta\in\Psef(X)$, and assume $E_1,\dots,E_r\subset X$ are prime divisors such that $\ord_{E_i}(\theta)>0$ for all $i$. Then $[E_1],\dots,[E_r]$ are linearly independent in $\Num(X)$. In particular, $r\le\rho(X)=\dim\Num(X)$. 
\end{lem}
\begin{proof} We reproduce the simple argument of~\cite[Theorem~3.5~(v)]{BDPP} for the convenience of the reader. 
  By~\eqref{equ:vpsef}, after adding to $\theta$ a small enough ample class we assume that $\theta$ is big. Suppose $\sum_i c_i [E_i]=0$ with $c_i\in\R$, so that $G:=\sum_i c_i E_i$ is numerically equivalent to $0$, and choose $0<\e\ll 1$ such that $\ord_{E_i}(\theta)+\e c_i>0$ for all $i$.
  Pick any effective $\R$-divisor $D\equiv\theta$ and set $D':=D+\e G$. Then $D'$ is effective, since 
$$
\ord_{E_i}(D')=\ord_{E_i}(D)+\e c_i\ge\ord_{E_i}(\theta)+\e c_i>0
$$
 for all $i$. Since $G\equiv 0$, we also have $D'\equiv\theta$, and~\eqref{equ:vnum} thus yields for each $i$  
 $$
 \ord_{E_i}(\theta)\le\ord_{E_i}(D')=\ord_{E_i}(D)+\e c_i. 
$$ 
Taking the infimum over $D$ we get $\ord_{E_i}(\theta)\le\ord_{E_i}(\theta)+\e c_i$ (see Proposition~\ref{prop:Vmin}~(ii)), \ie $c_i\ge 0$ for all $i$. Thus $G\ge 0$, and hence $G=0$, since $G\equiv 0$. This proves $c_i=0$ for all $i$ which shows, as desired, that the $[E_i]$ are linearly independent. 
\end{proof} 

\begin{proof}[Proof of Theorem~\ref{thm:Vdivtype}] By~\eqref{equ:vV}, the first assertion means that there are only finitely many prime divisors $E\subset X$ such that $\ord_E(\theta)>0$, and is thus a direct consequence of Lemma~\ref{lem:finitediv}. Pick $v\in X^\div$. The second point is equivalent to $v(\theta)>0\Leftrightarrow c_X(v)\in\B_-(\theta)$. When $\theta$ is big, this is the content of~\cite[Theorem~B]{ELMNP}. In the general case, pick $\om\in\Amp(X)$. Then $v(\theta)>0$ iff $v(\theta+\e\om)>0$ for $0<\e\ll 1$, by~\eqref{equ:vpsef}, while $c_X(v)\in\B_-(\theta)$ iff $c_X(v)\in\B_-(\theta+\e\om)$ for $0<\e\ll 1$, by~\eqref{equ:Bminplus}. The result follows. 
\end{proof}
For later use, we also note: 
\begin{lem}\label{lem:minpoly} For any polyhedral subcone $C\subset\Psef(X)$, we have: 
\begin{itemize}
\item[(i)] $\theta\mapsto v(\theta)$ is continuous on $C$ for all $v\in X^\div$; 
\item[(ii)] the set of prime divisors $E\subset X$ such that $\ord_E(\theta)>0$ for some $\theta\in C$ is finite.
\end{itemize}
\end{lem}
\begin{proof} As mentioned in Remark~\ref{rmk:mincont}, any convex, lsc function on a polyhedral cone is continuous (see~\cite{How}), and (i) follows. To see (ii), pick a finite set of generators $(\theta_i)$ of $C$. Each $\theta\in C$ can be written as $\theta=\sum_i t_i \theta_i$ with $t_i\ge 0$. By convexity and homogeneity of minimal vanishing orders, this implies $\ord_E(\theta)\le\sum_i t_i\ord_{E}(\theta_i)$, so that $\ord_E(\theta)>0$ implies $\ord_E(\theta_i)>0$ for some $i$. The result now follows from Lemma~\ref{lem:finitediv}. 
\end{proof}
%
%
\section{Zariski decompositions}\label{sec:Zariski}
Next we study the close relationship between the extremal function in~\S\ref{sec:extremal}, and the various versions of the Zariski decomposition of a psef numerical class.
%
%
\subsection{The $b$-divisorial Zariski decomposition}
Pick $\theta\in\Num(X)$ a psef class.
By Theorem~\ref{thm:Vdivtype}, the function $X^\div\ni v\mapsto v(\theta)=-V_\theta(v)$ is of divisorial type. We denote by 
$$
\Nz(\theta)\in\Zunb(X)_\R
$$
the corresponding effective $b$-divisor, which thus satisfies
$$
\p_{\Nz(\theta)}(v)=v(\Nz(\theta))=v(\theta)=-V_\theta(v) 
$$
for all $v\in X^\div$. This construction is birationally invariant: 

\begin{thm}\label{thm:bZarvar} For any $\theta\in\Psef(X)$, the $b$-divisor class 
$$
\env(\theta):=\overline{\theta}-[\Nz(\theta)]\in\Numb(X)
$$ 
is nef, and $\Nz(\theta)$ is the smallest effective $b$-divisor with this property. Moreover,
  \begin{equation}\label{equ:negZar}
    \Nz(\theta)\ge\overline{\Nz(\theta)_Y}
  \end{equation}
  for all birational models $Y\to X$. 
\end{thm}
We call $\overline{\theta}=\env(\theta)+[\Nz(\theta)]$ the \emph{$b$-divisorial Zariski decomposition} of $\theta$. At least when $\theta$ is big, this construction is basically equivalent to~\cite[Theorem~D]{KM}, and to the case $p=1$ of~\cite[\S 2.2]{diskant}. 

Note that the $b$-divisorial Zariski decomposition is birationally invariant:
\begin{lem}\label{lem:Zarbir} For any $\theta\in\Psef(X)$ and any birational model $\pi\colon Y\to X$, we have
  \[
    \Nz(\pi^\star\theta)=\Nz(\theta)
    \quad\text{and}\quad
    \env(\pi^\star\theta)=\env(\theta)
  \]
  in   $\Zunb(X)_\R=\Zunb(Y)_\R$ and $\Numb(X)_\R=\Numb(Y)_\R$, respectively.
\end{lem}
\begin{proof} Since $\PSH(\pi^\star\theta)=\pi^\star\PSH(\theta)$, see~\eqref{equ:pshbir}, we have $V_{\pi^\star\theta}=\pi^\star V_\theta$, and the result follows.
\end{proof}

\begin{proof}[Proof of Theorem~\ref{thm:bZarvar}] Since $\p_{-\Nz(\theta)}=V_\theta$ is $\theta$-psh, Proposition~\ref{prop:homogb} shows that $\overline{\theta}-[\Nz(\theta)]$ is nef, which yields the last point, by the Negativity Lemma (see Lemma~\ref{lem:neg}). Conversely, if $E\in\Zunb(X)_\R$ is effective with $\overline{\theta}-[E]$ nef, then $-\p_E\in\PSH_\hom(\theta)$, again by Proposition~\ref{prop:homogb}. Thus $-\p_E\le V_\theta=-\p_{\Nz(\theta)}$, and hence $E\ge\Nz(\theta)$.
\end{proof}
As a consequence of  Proposition~\ref{prop:Vmin}, we get
\begin{cor}\label{cor:Zarconv}
  The map $\Psef(X)\ni\theta\mapsto\Nz(\theta)\in\Zunb(X)$ is homogeneous, lsc, and convex.
\end{cor}
%
\subsection{The divisorial Zariski decomposition}
For any $\theta\in\Psef(X)$, we denote by $\Nz_X(\theta):=\Nz(\theta)_X$ the incarnation of $\Nz(\theta)\in\Zunb(X)_\R$ on $X$, which thus satisfies 
\begin{equation}\label{equ:NX}
\Nz_X(\theta)=\sum_{E\subset X} \ord_E(\theta) E
\end{equation}
with $E$ ranging over all prime divisors of $X$, and $\ord_E(\theta)=0$ for all but finitely many $E$. 

For any effective $\R$-divisor $D$ on $X$ with numerical class $[D]\in\Psef(X)$, \eqref{equ:vnum} yields
\begin{equation}\label{equ:Neff}
\Nz_X(D):=\Nz_X([D])\le D.
\end{equation}
More generally, the following variational characterization holds. 

\begin{thm}\label{thm:Zarvar} For any $\theta\in\Psef(X)$, the class 
$$
\env_X(\theta):=\theta-[\Nz_X(\theta)]\in\Num(X)
$$
is movable, and $\Nz_X(\theta)$ is the smallest effective $\R$-divisor on $X$ with this property. 
\end{thm}
Following~\cite{Bou}, we call the decomposition
$$
\theta=\env_X(\theta)+[\Nz_X(\theta)]
$$
the \emph{divisorial Zariski decomposition} of $\theta$. It coincides with the \emph{$\sigma$-decomposition} of~\cite{Nak}.

\begin{proof}[Proof of Theorem~\ref{thm:Zarvar}] By definition, $\env_X(\theta)$ is the incarnation on $X$ of $\overline{\theta}-[\Nz(\theta)]$. By Theorem~\ref{thm:bZarvar}, the latter class is nef, and $\env_X(\theta)$ is thus movable, by Lemma~\ref{lem:nefmov}. 

 To prove the converse, assume first that $\theta$ is movable. We then need to show $\Nz_X(\theta)=0$, \ie $\ord_E(\theta)=0$ for each $E\subset X$ prime (see~\eqref{equ:NX}). By~\eqref{equ:vnum}, this is clear if $\theta=c_1(L)$ for a big line bundle $L$ with base locus of codimension at least $2$. Since the movable cone $\Mov(X)$ is generated by the classes of such line bundles, the continuity of $\theta\mapsto\ord_E(\theta)$ on the big cone yields the result when $\theta$ is further big, and the case of an arbitrary movable class follows by~\eqref{equ:vpsef}.  

Finally, consider any $\theta\in\Psef(X)$ and any effective $\R$-divisor $D$ on $X$ such that $\theta-[D]$ is movable. For any $E\subset X$ prime we then have $\ord_E(\theta-[D])=0$ by the previous step, and $\ord_E([D])\le\ord_E(D)$ by~\eqref{equ:Neff}). Thus
$$
\ord_E(\theta)\le\ord_E(\theta-[D])+\ord_E(D)=\ord_E(D). 
$$
This shows $\Nz_X(\theta)\le D$, which concludes the proof. 
\end{proof}

\begin{rmk}\label{rmk:negmov} Theorem~\ref{thm:Zarvar} implies the following converse of Lemma~\ref{lem:nefmov}: a class $\theta\in\Num(X)$ is movable iff $\theta=\a_X$ for a nef $b$-divisor class $\a\in\Nefb(X)$. 
\end{rmk}

\begin{cor}\label{cor:restrpsef} Pick $\theta\in\Psef(X)$ and a prime divisor $E\subset X$. Then $(\theta-\ord_E(\theta)E)|_E\in\Num(E)
$ is pseudoeffective. 
\end{cor}
\begin{proof} We have $\theta-\ord_E(\theta)[E]=\env_X(\theta)+\sum_{F\ne E}\ord_{F}(\theta)[F]$, where $F$ ranges over all prime divisors of $X$ distinct from $E$. Since $\env_X(\theta)$ is movable, $\env_X(\theta)|_E$ is psef. On the other hand, $[F]|_E$ is psef for any $F\ne E$, and the result follows. 
\end{proof}

\begin{lem}\label{lem:bZarCar} For any $\theta\in\Psef(X)$ and any birational model $\pi\colon Y\to X$, the incarnation of $\Nz(\theta)$ on $Y$ coincides with $\Nz_Y(\pi^\star\theta)$. Further, the following are equivalent:
\begin{itemize}
\item[(i)] the $b$-divisor $\Nz(\theta)$ is $\R$-Cartier, and determined on $Y$;
\item[(ii)] $\env_Y(\pi^\star\theta)$ is nef.
\end{itemize}
\end{lem}
\begin{proof} The first point follows from Lemma~\ref{lem:Zarbir}. If (i) holds then the nef $b$-divisor $\overline{\theta}-\Nz(\theta)$ is $\R$-Cartier and determined on $Y$. Thus $(\overline{\theta}-\Nz(\theta))_Y=\pi^\star\theta-\Nz_Y(\pi^\star\theta)=\env_Y(\pi^\star\theta)$ is nef, and hence (i)$\Rightarrow$(ii). 

Conversely, assume (ii). Then $\overline{\Nz(\theta)_Y}=\overline{\Nz_Y(\pi^\star\theta)}$ is an effective $b$-divisor, and the $b$-divisor class $\overline{\theta}-[\overline{\Nz(\theta)_Y}]=\overline{\env_Y(\pi^\star\theta)}$ is nef. By Theorem~\ref{thm:bZarvar} this implies $\Nz(\theta)\le\overline{\Nz(\theta)_Y}$, while $\Nz(\theta)\ge\overline{\Nz(\theta)_Y}$ always holds (see~\eqref{equ:negZar}). This proves (ii)$\Rightarrow$(i). 
\end{proof}

Since any movable class on a surface is nef, we get: 
\begin{cor}\label{cor:bZarsurf} If $\dim X=2$ then $\Nz(\theta)=\overline{\Nz_X(\theta)}$ for all $\theta\in\Psef(X)$. 
\end{cor}
In contrast, see~\cite[Theorem~IV.2.10]{Nak} for an example of a big line bundle $L$ on a $4$-fold $X$ such that the $b$-divisor $\Nz(L)$ is not $\R$-Cartier, \ie $\env_Y(\pi^\star L)$ is not nef for any model $\pi\colon Y\to X$.

%
\subsection{Zariski exceptional divisors and faces}
This section revisits~\cite[\S 3.3]{Bou}. 

\begin{defi}\label{defi:Zarexc} We say that:
\begin{itemize}
\item[(i)] an effective $\R$-divisor $D$ on $X$ is \emph{Zariski exceptional} if $\Nz_X(D)=D$, or equivalently, $\env_X([D])=0$; 
\item[(ii)] a finite family $(E_i)$ of prime divisors $E_i\subset X$ is \emph{Zariski exceptional} if every effective $\R$-divisor supported in the $E_i$'s is Zariski exceptional. 
\end{itemize}
We also define a \emph{Zariski exceptional face} $F$ of $\Psef(X)$ as an extremal subcone such that $\env_X|_F\equiv 0$. 
\end{defi}
Here a closed subcone $C\subset\Psef(X)$ is extremal iff $\a,\b\in\Psef(X)$, $\a+\b\in C$ implies $\a,\b\in C$.

We first note: 
\begin{lem}\label{lem:bZex} An effective $\R$-divisor $D$ is Zariski exceptional iff $\Nz(D)=\overline{D}$. 
\end{lem}
\begin{proof} Assume $\Nz_X(D)=D$. Then $\Nz(D)\le\overline{D}$, by Theorem~\ref{thm:bZarvar}, and $\Nz(D)\ge\overline{\Nz_X(D)}=\overline{D}$ (see~\eqref{equ:negZar}). The result follows. 
\end{proof}

The above notions are related as follows: 

\begin{thm}\label{thm:Zarexc} The properties following hold:
\begin{itemize}
\item[(i)] if $E\subset X$ is a prime divisor, then $E$ is either movable (in which case $E|_E$ is psef), or it is Zariski exceptional; 
\item[(ii)] the set of Zariski exceptional families of prime divisors on $X$ is at most countable;
\item[(iii)] for any $\theta\in\Psef(X)$, the irreducible components of $\Nz_X(\theta)$ form a Zariski exceptional family; in particular, $\Nz_X(\theta)$ is Zariski exceptional; 
\item[(iv)] each Zariski exceptional family $(E_i)$ is linearly independent in $\Num(X)$, and generates a Zariski exceptional face $F:=\sum_i \R_{\ge 0}[E_i]$ of $\Psef(X)$; 
\item[(v)] conversely, each Zariski exceptional face $F$ of $\Psef(X)$ arises as in (iv).
\end{itemize}
\end{thm}

\begin{proof} Assume $E\subset X$ is prime. Then $\Nz_X(E)\le E$ (see~\eqref{equ:Neff}), and hence $\Nz_X(E)=c E$ with $c\in [0,1]$. If $c=1$, then $E$ is Zariski exceptional. Otherwise, 
$$
E=(1-c)^{-1}(E-\Nz_X(E))\equiv(1-c)^{-1}\env_X(E)
$$
is movable. This proves (i). 

To see (ii), note that a Zariski exceptional prime divisor satisfies $E=\Nz_X(E):=\Nz_X([E])$, and hence is uniquely determined by its numerical class $[E]\in\Num(X)_\Q$. As a consequence, the set of Zariski exceptional primes is at most countable, and hence so is the set of Zariski exceptional families. 

Pick $\theta\in\Psef(X)$. We first claim that $D:=\Nz_X(\theta)$ is Zariski exceptional. Since $\env_X(\theta)=\theta-[D]$ and $\env_X(D)=[D-\Nz_X(D)]$ are both movable, $\theta-[\Nz_X(D)]$ is movable as well. Theorem~\ref{thm:Zarvar} thus yields $\Nz_X(D)\ge\Nz_X(\theta)=D$, which proves the claim in view of~\eqref{equ:Neff}. Denote by $D=\sum_{i=1}^r c_i E_i$ the irreducible decomposition of $D$, and set $f_i(x):=\ord_{E_i}(\sum_j x_j E_j)$ for $1\le i\le r$. This defines a convex function $f_i\colon\R_{\ge 0}^r\to\R_{\ge 0}$ which satisfies $f_i(x)\le x_i$ for all $x$, by~\eqref{equ:Neff}. Since equality holds at the interior point $x=c\in\R_{>0}^r$, we necessarily have $f_i(x)=x_i$ for all $x\in\R_{\ge 0}^r$, which proves (iii). 

Next pick a Zariski exceptional family $(E_i)$. By Lemma~\ref{lem:finitediv}, the $[E_i]$ are linearly independent in $\Num(X)$. By definition, we have $\env_X\equiv 0$ on $F:=\sum_i\R_{\ge 0}[E_i]$. To see that $F$ is an extremal face of $\Psef(X)$, pick $D:=\sum_i c_i E_i$ with $c_i\ge 0$, and assume $[D]=\a+\b$ with $\a,\b\in\Psef(X)$. We need to show that both $\a$ and $\b$ lie in $F$. By Definition~\ref{defi:Zarexc} we have $D=\Nz_X(D)\le\Nz_X(\a)+\Nz_X(\b)$, and hence 
\begin{multline}
  [\Nz_X(\a)]+[\Nz_X(\b)]\le\env_X(\a)+\env_X(\b)+[\Nz_X(\a)]+[\Nz_X(\b)]\\
  =\a+\b=[D]\le[\Nz_X(\a)]+\Nz_X(\b)], 
\end{multline}
with respect to the psef order on $\Num(X)$. Since $\Psef(X)$ is strict, we infer $\env_X(\a)=\env_X(\b)=0$ and $[D]=[\Nz_X(\a)]+[\Nz_X(\b)]$. Since $\Nz_X(\a)+\Nz_X(\b)-D$ is effective, it follows that $\Nz_X(\a)+\Nz_X(\b)=D$. This implies that $\Nz_X(\a)$ and $\Nz_X(\b)$ are supported in the $E_i$'s, which proves, as desired, that $\a=[\Nz_X(\a)]$ and $\b=[\Nz_X(\b)]$ both lie in $F$. Thus (iv) holds.

Conversely, assume that $F\subset\Psef(X)$ is a Zariski exceptional face, and pick a class $\theta$ in its relative interior $\mathring{F}$. By (iii), the components $(E_i)$ of $\Nz_X(\theta)$ form a Zariski exceptional family, which thus generates a Zariski exceptional face $F':=\sum_i\R_{\ge 0}[E_i]$. Since $F$ and $F'$ are both extremal faces containing $\theta$ in their relative interior, we conclude $F=F'$, which proves (v). 
\end{proof}

As a result, Zariski exceptional families are in 1--1 correspondence with Zariski exceptional faces, which are rational simplicial cones generated by Zariski exceptional primes.


For surfaces, we recover the classical picture (see \eg~\cite[\S 4]{Bou}): 

\begin{thm}\label{thm:Zarsurf} Assume $\dim X=2$. Then:
\begin{itemize}
\item[(i)] a finite family $(E_i)$ of prime divisors on $X$ is Zariski exceptional iff the intersection matrix $(E_i\cdot E_j)$ is negative definite;
\item[(ii)] for any $\theta\in\Psef(X)$, $\theta=\env_X(\theta)+[\Nz_X(\theta)]$ coincides with the classical Zariski decomposition, \ie $\env_X(\theta)$ is nef, $\Nz_X(\theta)$ is Zariski exceptional, and $\env_X(\theta)\cdot\Nz_X(\theta)=0$. 
\end{itemize}
\end{thm}

%
\subsection{Piecewise linear Zariski decompositions}
We introduce the following terminology:

\begin{defi} Given a convex subcone $C\subset\Psef(X)$, we say that \emph{the Zariski decomposition is piecewise linear} (\emph{PL} for short) on $C$ if the map $\Nz\colon C\to\Zunb(X)_\R$ extends to a PL map $\Num(X)\to\Zunb(X)_\R$, \ie a map that is linear on each cone of some finite fan decomposition of $\Num(X)$. If the fan and the linear maps on its cones can further be chosen rational, then we say that \emph{the Zariski decomposition is $\Q$-PL} on $C$. 
\end{defi}

\begin{lem}\label{lem:ZarPL} Let $C\subset\Psef(X)$ be a convex cone, and assume that $C$ is written as the union of finitely many convex subcones $C_i$. Then the Zariski decomposition is PL (resp.~$\Q$-PL) on $C$ iff it is PL (resp.~$\Q$-PL) on each $C_i$. 
\end{lem}
\begin{proof} The `only if' part is clear. Conversely, assume the Zariski decomposition is PL (resp.~$\Q$-PL) on each $C_i$. After further subdividing each $C_i$ according to a fan decomposition of $\Num(X)$, we may assume that there exists a linear (resp.~rational linear) map $L_i\colon\Num(X)\to\Zunb(X)_\R$ that coincides with $\Nz$ on $C_i$. If $C_i$ has nonempty interior in $C$, then $L_i|_{\Vect C}$ is uniquely determined as the derivative of $\Nz$ at any interior point of $C_i$, and we have $\Nz\ge L_i$ on $C$ by convexity of $\Nz$, see Corollary~\ref{cor:Zarconv}. Set $F:=\max_i L_i$, where the maximum is over all $C_i$ with nonempty interior in $C$. Then $F\colon\Num(X)\to\Zunb(X)_\R$ is PL (resp.~$\Q$-PL), $\Nz\ge F$ on $C$, and equality holds outside the union $A$ of all $C_i$ with empty interior in $C$. Since $A$ has zero measure, its complement is dense in $C$. Since $\Nz-F$ is lsc, see Corollary~\ref{cor:Zarconv}, we infer $\Nz\le F$ on $C$, which proves the `if' part. 
\end{proof}

As a consequence of~\cite[Theorem~4.1]{ELMNP}, we have: 

\begin{exam}\label{exam:ZarPL} If $X$ is a Mori dream space (\eg of log Fano type), then:
\begin{itemize}
\item for each $\theta\in\Psef(X)$, the $b$-divisor $\Nz(\theta)$ is $\R$-Cartier;
\item $\Psef(X)$ is a rational polyhedral cone;
\item the Zariski decomposition is $\Q$-PL on $\Psef(X)$. 
\end{itemize}
\end{exam}

The next result is closely related to the theory of Zariski chambers studied in~\cite{BKS}. 

\begin{prop}\label{prop:ZarPL} If $\dim X=2$, then the Zariski decomposition is $\Q$-PL on any convex cone $C\subset \Psef(X)$ with the property that the set of prime divisors $E\subset X$ with $\ord_E(\theta)>0$ for some $\theta\in C$ is finite. 
\end{prop}
By Lemma~\ref{lem:minpoly}~(ii), the finiteness condition on $C$ is satisfied as soon as $C$ is polyhedral. 

\begin{proof} For each Zariski exceptional face $F$ of $\Psef(X)$ with relative interior $\mathring{F}$, set $Z_F:=\Nz_X^{-1}(\mathring{F})$. Thus $\theta\in\Psef(X)$ lies in $Z_F$ iff the irreducible decomposition of $\Nz_X(\a)$ are precisely the generators of $F$. By Theorem~\ref{thm:Zarsurf}~(ii), $Z_F$ is a convex subcone of $\Psef(X)$ (whose intersection with $\Bg(X)$ is a Zariski chamber in the sense of~\cite{BKS}); further, $\Nz_X|_{Z_F}:Z_F\to \mathring{F}$ is the restriction of the orthogonal projection onto $\Vect F$, which is a rational linear map. By Corollary~\ref{cor:bZarsurf}, the Zariski decomposition is thus $\Q$-PL on $Z_F$. Finally, the finiteness assumption guarantees that $C$ meets only finitely many $Z_F$'s, and the result is thus a consequence of Lemma~\ref{lem:ZarPL}. 
\end{proof}

We conclude this section with a higher-dimensional situation in which Zariski decompositions can be analyzed. Assuming again that $\dim X$ is arbitrary, consider next a $2$-dimensional cone $C\subset\Num(X)$ generated by two classes $\theta,\a\in\Num(X)$ such that $\theta\in\Nef(X)$ and $\a\notin\Psef(X)$. Set 
$$
C_\nef:=C\cap\Nef(X)\subset C_\psef:=C\cap\Psef(X)\subset C, 
$$
and introduce the thresholds 
$$
\la_\nef:=\sup\{\la\ge 0\mid\theta+\la\a\in\Nef(X)\},\quad \la_\psef:=\sup\{\la\ge 0\mid\theta+\la \a\in\Psef(X)\}, 
$$
so that $C_\nef$ (resp.~$C_\psef$) is generated by $\theta$ and $\theta_\nef:=\theta+\la_\nef \a$ (resp~$\theta_\psef:=\theta+\la_\psef\a$). 

The next result is basically contained in~\cite[\S6.5]{Per}.
\begin{prop}\label{prop:Zar2D} With the above notation, suppose that $C$ contains the class of a prime divisor $S\subset X$ such that $\Nef(S)=\Psef(S)$ and $S|_S$ is not nef. Then:
\begin{itemize}
\item[(i)] $\theta_\psef=t[S]$ with $t>0$;
\item[(ii)] $\la_\nef=\la^S_\nef:=\sup\left\{\la\ge 0\mid (\theta+\la\a)|_S\in\Nef(S)\right\}$;
\item[(iii)] the Zariski decomposition is PL on $C_\psef$, with 
$$
\Nz\equiv 0\text{ on }C_\nef,\quad\Nz(a\theta_\nef+b [S])=b\overline S\text{ for all }a,b\ge 0. 
$$
\end{itemize}
\end{prop}

\begin{proof} The assumptions imply that $S|_S$ is not psef. By Theorem~\ref{thm:Zarexc}~(i), $S$ is thus Zariski exceptional, and $[S]$ generates an extremal ray of $\Psef(X)$. This ray is also extremal in $C_\psef$, which proves (i). 

Next, note that $\la_\nef\le\la^S_\nef\le\la_\psef$, by (i). Pick a curve $\g\subset X$. We need to show $(\theta+\la^S_\nef\a)\cdot\g\ge 0$. This is clear if $\g\subset S$ (since $(\theta+\la^S_\nef\a)|_S$ is nef), or if $\a\cdot\g\ge 0$ (since $\theta\cdot\g\ge 0$ and $\la^S_\nef\ge 0$). Otherwise, we have $S\cdot\g\ge 0$ and $\a\cdot\g\le 0$, and we get again $(\theta+\la^S_\nef \a)\cdot\g\ge 0$ since 
$$
\theta+\la^S_\nef \a\equiv \theta_\psef+(\la^S_\nef-\la_\psef)\a= t[S]+(\la^S_\nef-\la_\psef) \a
$$
with $\la^S_\nef-\la_\psef\le 0$. This proves (ii). 

For~(iii), note that $\Nz\equiv 0$ on $\Nef(X)\supset C_\nef$ (see Theorem~\ref{thm:bZarvar}). Further, $\Nz([S])=\overline S$ (see Lemma~\ref{lem:bZex}), and hence $\Nz(a\theta_\nef+b[S])\le b\overline S$ for $a,b\ge 0$. In particular, $c:=\ord_S(a\theta_\nef+b[S])\le b$. On the other hand, \eqref{equ:negZar} yields 
$$
\Nz(a\theta_\nef+b[S])\ge\overline{\Nz(a\theta_\nef+b[S])}\ge c\overline S, 
$$
and it thus remains to see $c=b$. By Corollary~\ref{cor:restrpsef}, 
$\left((a\theta_\nef+b[S])-c[S]\right)|_S$ lies in $\Psef(S)=\Nef(S)$. By (b), we infer $a\theta_\nef+(b-c)[S]\in C_\nef$, and hence $b-c=0$, since $C_\nef=\R_{\ge 0}\theta+\R_{\ge 0}\theta_\nef$ intersects $\R_{\ge 0}\theta_\nef+\R_{\ge 0}[S]$ only along $\R_{\ge 0}\theta_\nef$. 
\end{proof}

%
%
\section{Green's functions and Zariski decompositions}\label{sec:greenzar}
In this section we fix an ample class $\om\in\Amp(X)$. 
%
%
\subsection{Green's functions and equilibrium measures}\label{sec:equmeas}
A subset $\Sigma\subset X^\an$ is \emph{pluripolar} if $\Sigma\subset\{\f=-\infty\}$ for some $\f\in\PSH(\om)$. By~\cite[Theorem~4.5]{trivval}, $\Sigma$ is nonpluripolar iff
\begin{equation*}
\te(\Sigma):=\sup_{\f\in\PSH(\om)}(\sup\f-\sup_\Sigma\f)\in [0,+\infty]
\end{equation*}
is finite. The invariant $\te(\Sigma)$, which plays an important role in~\cite{BlJ,nakstab1}, is modeled in the Alexander-Taylor capacity (which corresponds to  $e^{-\te(\Sigma)}$)
in complex analysis.

\begin{defi}\label{defi:extpp} For any subset $\Sigma\subset X^\an$ we set 
\begin{equation}\label{equ:Vsig}
\f_\Sigma=\f_{\om,\Sigma}:=\sup\{\f\in\PSH(\om)\mid \f|_{\Sigma}\le 0\}.
\end{equation}
\end{defi}
Note that $\f_\Sigma(v_\triv)=\sup\f_\Sigma=\te(\Sigma)$, and hence 
\begin{equation}\label{equ:SigPL}
\f_\Sigma\in\PL(X)\Longrightarrow\te(\Sigma)\in\Q.
\end{equation}

\begin{thm}\label{thm:extpp} For any compact subset $\Sigma\subset X^\an$, the following holds: 
\begin{itemize}
\item[(i)] $\f_\Sigma=\sup\{\f\in\CPSH(\om)\mid\f|_\Sigma\le 0\}$; in particular, $\f_\Sigma$ is lsc; 
\item[(ii)] if $\Sigma$ is pluripolar then $\f_\Sigma^\star\equiv+\infty$; 
\item[(iii)] if $\Sigma$ is nonpluripolar, then $\f_\Sigma^\star$ is $\om$-psh and nonnegative; further, $\mu_\Sigma:=\MA(\f_\Sigma^\star)$ is supported in $\Sigma$, $\int\f_\Sigma^\star\,\mu_\Sigma=0$, and $\mu_\Sigma$ is characterized as the unique minimizer of the energy $\|\mu\|$ over all Radon probability measures $\mu$ with support in $\Sigma$. 
\end{itemize}
\end{thm} 

Since the energy of a Radon probability measure $\mu$ only appears in this statement, we simply recall here that it is defined as 
\begin{equation}\label{equ:enmes}
\|\mu\|=\sup_{\f\in\cE^1(\om)}\left(\en(\f)-\int\f\,\mu\right)\in [0,+\infty], 
\end{equation}
and refer to~\cite[\S 9.1]{trivval} for more details. 

\begin{defi} Assuming $\Sigma$ is nonpluripolar, we call $\mu_\Sigma$ its \emph{equilibrium measure}, and $\f_\Sigma^\star$ its \emph{Green's function}. 
\end{defi}
The latter is characterized as the normalized potential of $\mu_\Sigma$ (in the terminology of~\cite[\S 1.6]{nakstab2}), \ie the unique $\f\in\cE^1(\om)$ such that $\MA(\f)=\mu_\Sigma$ and $\int\f\,\mu_\Sigma=0$.

\begin{proof}[Proof of Theorem~\ref{thm:extpp}] Denote by $\f'_\Sigma$ the right-hand side in (i), which obviously satisfies $\f'_\Sigma\le\f_\Sigma$. Pick $\f\in\PSH(\om)$ with $\f|_\Sigma\le 0$, and write $\f$ as the limit of a decreasing net $(\f_i)$ in $\CPSH(\om)$. For any $\e>0$, a Dini type argument shows that $\f_i<\e$ on $\Sigma$ for $i$ large enough. Thus $\f_i\le\f'_\Sigma+\e$, and hence $\f\le\f'_\Sigma+\e$. This shows $\f_\Sigma\le\f'_\Sigma$, which proves (i). 

Next, (ii) and the first half of (iii) follow from~\cite[Lemma~13.15]{trivval}. Since the negligible set $\{\f_\Sigma<\f_\Sigma^\star\}$ is pluripolar (see~\cite[Theorem~13.17]{trivval}), it has zero measure for any measure $\mu$ of finite energy~\cite[Lemma~9.2]{trivval}. If $\mu$ has support in $\Sigma$, this yields $\int \f_\Sigma^\star\,\mu=\int \f_\Sigma\,\mu=0$. By~\eqref{equ:enmes} we infer $\|\mu\|\ge\en(\f_\Sigma^\star)=\|\mu_\Sigma\|$. This proves that $\mu_\Sigma$ minimizes the energy, while uniqueness follows from the strict convexity of the energy~\cite[Proposition~10.10]{trivval}. 
\end{proof}

Further mimicking classical terminology in the complex analytic setting, we introduce: 

\begin{defi} We say that a compact subset $\Sigma\subset X^\an$ is \emph{regular} if $\f_\Sigma\in\CPSH(\om)$. 
\end{defi}
In particular, $\Sigma$ is nonpluripolar (see Theorem~\ref{thm:extpp}).  
\begin{lem}\label{lem:reg} For any compact subset $\Sigma\subset X^\an$, the following hold: 
\begin{itemize}
\item[(i)] $\Sigma$ is regular iff $\f_\Sigma^\star\le 0$ on $\Sigma$; 
\item[(ii)] the regularity of $\Sigma$ is independent of $\om\in\Amp(X)$; 
\item[(iii)] if $\Sigma\subset X^\lin$ then $\Sigma$ is regular.
\end{itemize}
\end{lem}
\begin{proof} If $\Sigma$ is regular, then $\f_\Sigma^\star=\f_\Sigma$ vanishes on $\Sigma$. Conversely, assume $\f^\star_\Sigma\le 0$ on $\Sigma$. By (ii) and (iii) of Theorem~\ref{thm:extpp}, $\Sigma$ is necessarily nonpluripolar, and $\f_\Sigma^\star$ is $\om$-psh. It is thus a competitor in~\eqref{equ:Vsig}, which implies that $\f_\Sigma=\f_\Sigma^\star$ is $\om$-psh, and also continuous by Theorem~\ref{thm:extpp}~(i). 

Assume $\Sigma$ is regular for $\om$, and pick $\om'\in\Amp(X)$. Then $t\om-\om'$ is nef for $t\gg 1$, and hence $\PSH(\om')\subset t\PSH(\om)$. This implies $\f_{\om',\Sigma}\le t \f_{\om,\Sigma}$, and hence $\f^\star_{\om',\Sigma}\le t \f_{\om,\Sigma}$. In particular, $\f^\star_{\om',\Sigma}|_{\Sigma}\le 0$, which proves that $\Sigma$ is regular for $\om'$, by (i). 

Finally, assume $\Sigma\subset X^\lin$. Since $\{\f_\Sigma<\f_\Sigma^\star\}$ is pluripolar (see~\cite[Theorem~13.17]{trivval}), it is disjoint from $X^\lin$. As a result, $\f_\Sigma^\star\in\PSH(\om)$ vanishes on $\Sigma$, and it again follows from (i) that $\Sigma$ is regular. 
\end{proof}

%
\subsection{The Green's function of a real divisorial set}
In what follows, we consider a \emph{real divisorial set}, by which we mean a finite set $\Sigma\subset X^\div_\R$ of real divisorial valuations. By Lemma~\ref{lem:reg}~(iii), $\Sigma\subset X^\lin$ is regular, \ie $\f_\Sigma\in\CPSH(\om)$. When $\Sigma=\{v\}$ for a single $v\in X^\div_\R$, we simply write $\f_v:=\f_\Sigma$. 

\begin{exam}\label{exam:dreamy} Assume $\om=c_1(L)$ with $L\in\Pic(X)_\Q$ ample and $v\in X^\div$. Then $v$ is \emph{dreamy} (with respect to $L$) in the sense of K.Fujita iff $\f_v\in\cH(L)$; see~\cite[\S1.7,Appendix~A]{nakstab1}. 
\end{exam} 

If $v_\triv\in\Sigma$, then $\f_\Sigma\equiv 0$, and we henceforth assume $v_\triv\notin\Sigma$. Pick a smooth birational model $\pi\colon Y\to X$ which extracts each $v\in\Sigma$, \ie $v=t_v\ord_{E_v}$ for a prime divisor $E_v\subset Y$ and $t_v\in\R_{>0}$. We then introduce the effective $\R$-divisor on $Y$
$$
D:=\sum_\a t_\a^{-1} E_\a, 
$$
whose set of Rees valuations $\Ga_D$ coincides with $\Sigma$ (see Definition~\ref{defi:ReesR}). 

\begin{thm}\label{thm:Greenb} With the above notation, the following holds:
\begin{itemize}
\item[(i)] $\sup \f_\Sigma=\te(\Sigma)$ coincides with the pseudoeffective threshold
$$
\la_\psef:=\max\left\{\la\ge 0\mid\pi^\star\om-\la D\in\Psef(Y)\right\}; 
$$
\item[(ii)] $\f_{\Sigma}\in\CPSH(\om)$ is of divisorial type, and the associated family of $b$-divisors $(B_\la)_{\la\le\la_\psef}$ (see Theorem~\ref{thm:pshb}) is given by 
$$
-B_\la=\left\{
\begin{array}{ll}
\Nz(\pi^\star\om-\la D)+\la\overline{D} & \text{for }\la\in[0,\la_\psef]\\
0  & \text{for }\la\le 0.
\end{array}
\right. 
$$ 
\end{itemize}
\end{thm}

\begin{proof} Pick $\la\in\R$. For any $\p\in\PSH(\om)$, we have $\p+\la\le\f\Leftrightarrow\p|_\Sigma\le-\la$, and hence 
$$
\hf_\Sigma^\la=\sup\{\p\in\PSH_\hom(\om)\mid\p|_{\Sigma}\le-\la\}. 
$$
When $\la\le 0$ this yields $\hf_\Sigma^\la=0$. Assume now $\la>0$. Using Proposition~\ref{prop:homdom} and $\PSH_\hom(\pi^\star\om)=\pi^\star\PSH_\hom(\om)$, we get
\begin{equation}\label{equ:green1}
  \pi^\star\hf^\la_\Sigma
  =\sup\{\tau\in\PSH_\hom(\pi^\star\om-\la D)\}-\la\p_D=V_{\pi^\star\om-\la D}-\la\p_D. 
\end{equation}
Now the left-hand side is not identically $-\infty$ iff $\la\le\sup\f$, while for the right-hand side this holds iff $\la\le\la_\psef$, by Proposition~\ref{prop:extfun}. This proves (i), and also (ii), by Theorem~\ref{thm:Vdivtype}.
\end{proof}
\begin{cor}\label{cor:Greenactive} The center of $\f_\Sigma$ satisfies
$$
Z_X(\f_\Sigma)=\pi\left(\B_-(\pi^\star\om-\la_\psef D)\right)\cup Z_X(\Sigma).
$$
In particular, $Z_X(\f_\Sigma)$ is Zariski dense in $X$ iff $\B_-(\pi^\star\om-\la_\psef D)$ is Zariski dense in $Y$. 
\end{cor}
\begin{proof}
  By Lemma~\eqref{lem:centhom}, we have
  \[
    Z_X(\f_\Sigma)
    =Z_X(\hf_\Sigma^{\max})
    =\pi(Z_Y(\pi^*\hf_\Sigma^{\max})).
  \]
  It follows from Theorem~\ref{thm:Greenb} and its proof that
  \[
    \pi^\star\hf^{\max}_\Sigma=V_{\pi^\star\om-\la_\psef D}-\la_\psef\p_D.
  \]
  Now $Z_Y(V_{\pi^\star\om-\la_\psef D})=\B_-(\pi^\star\om-\la_\psef D)$ by 
  Theorem~\ref{thm:Vdivtype}, whereas we see from Example~\ref{exam:ideal} that $Z_Y(-\la_\psef\p_D)=Z_Y(\Sigma)$, so we conclude using Lemma~\ref{lem:centsum}.
\end{proof}
%
%
\subsection{Dimension one and two}\label{sec:Greensurf}
In this section we consider the case $\dim X\le 2$.
\begin{prop}\label{prop:greencurve1} If $\dim X=1$, then for any real divisorial set $\Sigma\subset X^\div_\R$, we have $\f_\Sigma\in\RPL^+(X)$. If $\om$ is rational and $\Sigma\subset X^\div$, then we further have $\f_\Sigma\in\PL^+(X)$.
\end{prop}
\begin{proof}
  We may assume $v_\triv\not\in\Sigma$, or else $\f_\Sigma\equiv0$. Thus assume
  $\Sigma=\{v_i\}_{i\in I}$, where $v_i=t_i\ord_{p_i}$, $t_i\in\R_{>0}$, and $p_i\in X$ is a closed point. We may assume $p_i\ne p_j$ for $i\ne j$, or else $\f_\Sigma=\f_{\Sigma'}$ for  $\Sigma'=\{v_i\}_{i\in I'}$, where $I'\subset I$ is defined by $i\in I'$ iff for all $j\ne i$, either $p_j\ne p_i$ or $t_j>t_i$.
  Under these assumptions,
  \[\f_\Sigma=A\max\{1+\sum_it_i^{-1}\log|\fm_{p_i}|,0\},
  \]
  where $A>0$ satisfies $A\sum_it_i^{-1}=\deg\om$, see~\cite[Example~3.19]{trivval}. Thus $\f_\Sigma\in\RPL^+(X)$. Further, if $\Sigma\subset X^\div$, then $t_i\in\Q_{>0}$ for all $i$, so if $\om$ is rational, then $A\in\Q_{>0}$, and hence $\f_\Sigma\in\PL^+(X)$.
\end{proof}
\begin{thm}\label{thm:Greensurf} If $\dim X=2$, then for any real divisorial set $\Sigma\subset X^\div_\R$, we have $\f_\Sigma\in\RPL^+(X)$. If $\om$ is rational and $\Sigma\subset X^\div$, then we further have 
\begin{equation}\label{equ:fPL1}
\f_\Sigma\in\PL(X)\Leftrightarrow\f_\Sigma\in\PL^+(X)\Leftrightarrow\te(\Sigma)\in\Q.
\end{equation}
\end{thm}
We will see in Example~\ref{exam:absurf2} that $\te(\Sigma)$ can be irrational.
\begin{lem}\label{lem:Lip} Assume $\dim X\le 2$, and pick $B\in\Carb(X)_\R$. Then $B$ is relatively nef iff it is relatively semiample.
\end{lem}
\begin{proof} Assume $B$ is relative nef, and pick a determination $\pi\colon Y\to X$ of $B$. The relatively nef cone of $\Num(Y/X)$ is dual to the cone generated by the (finite) set of $\pi$-exceptional prime divisors, and is thus a rational polyhedral cone. As a consequence, we can write $B_Y=\sum_i t_i D_i$ with $t_i>0$ and $D_i\in\Div(Y)_\Q$ relatively nef. By~\cite[Theorem~12.1~(ii)]{Lip}, each $D_i$ is relatively semiample, and the result follows. 
\end{proof}

\begin{proof}[Proof of Theorem~\ref{thm:Greensurf}] Use the notation of Theorem~\ref{thm:Greenb}. By Proposition~\ref{prop:ZarPL}, the Zariski decomposition is $\Q$-PL on the cone 
$$
C=(\R_+\pi^\star\om+\R_+[-D])\cap\Psef(Y)=\R_+\pi^\star\om+\R_+(\pi^\star\om-\la_\psef[D]).
$$
We can thus find $0=\la_1<\la_2<\dots<\la_N=\la_\psef$ such that 
$$
\la\mapsto B_\la=-(\Nz(\pi^\star\om-\la[D])+\la\overline{D})
$$ 
is affine linear on $[\la_i,\la_{i+1}]$ for $1\le i<N$. Setting $B_i:=B_{\la_i}$, it follows that 
$$
\f_\Sigma=\sup_{\la\in[0,\la_\psef]}\{\p_{B_\la}+\la\}=\max_{1\le i\le N}\{\p_{B_i}+\la_i\}. 
$$
Since $\overline{\om}+[B_i]$ is nef, the antieffective divisor $B_i$ is relatively nef, and hence relatively semiample (see Lemma~\ref{lem:Lip}). By Proposition~\ref{prop:CarPL}, we infer $\p_{B_i}\in\PL^+_\hom(X)_\R$, and hence $\f_\Sigma\in\RPL^+(X)$. 

Now assume $\om$ and $\te(\Sigma)=\la_\psef$ are both rational, and that $\Sigma\subset X^\div$. Then $D$ is rational as well, and $C$ is thus a rational polyhedral cone. Since the Zariski decomposition on $C$ is the restriction of a $\Q$-PL map on $\Num(Y)$, this implies that the $\la_i$ above can be chosen rational. Using again that the Zariski decomposition is $\Q$-PL on $C$, we infer that $B_i$ is a $\Q$-divisor, hence $\p_{B_i}\in\PL^+_\hom(X)$, which shows $\f_\Sigma\in\PL^+(X)$. The rest follows from~\eqref{equ:SigPL}. 
\end{proof}
%
%
\section{Examples of Green's functions}\label{sec:exgreen}
We now exhibit examples of Green's functions with various types of behavior. These examples serve as the underpinnings of Theorems~A and~B of the introduction.
%
%
\subsection{Divisors on abelian varieties}
As a direct application of Theorem~\ref{thm:Greenb}, we show: 

\begin{prop}\label{prop:nefpsef} Assume $\Nef(X)=\Psef(X)$. Consider a real divisorial set $\Sigma=\{v_\a\}\subset$ with $v_\a=t_\a\ord_{E_\a}$ for $E_\a\subset X$ prime, and set $D:=\sum_\a t_\a^{-1} E_\a$. Then 
$$
\te(\Sigma)=\la_\psef=\sup\left\{\la\ge 0\mid \om-\la D\in\Psef(X)\right\}
$$
and 
$$
\f_\Sigma=\te(\Sigma)\max\left\{0,1-\p_D\right\}.
$$
In particular, $\f_\Sigma\in\RPL^+(X)$. If we further assume $\Sigma\subset X^\div$, then 
\begin{equation}\label{equ:fPL2}
\f_\Sigma\in\PL(X)\Leftrightarrow\f_\Sigma\in\PL^+(X)\Leftrightarrow\te(\Sigma)\in\Q.
\end{equation}
\end{prop}
\begin{proof} Using the notation of Theorem~\ref{thm:Greenb}, we have $\Nz(\om-\la D)=0$ for $\la\le\la_{\psef}=\te(\Sigma)$. Thus $\hf_\Sigma^\la=-\la\p_D$, and hence 
$$
\f_\Sigma=\sup_{0\le\la\le\la_{\psef}}\{\la-\la\p_D\}=\la_{\psef}\max\left\{0,1-\p_D\right\}.
$$
Since $-\p_D=\sum_\a t_\a^{-1}\log|\cO_X(-E_\a)|$ lies in $\PL^+(X)_\R$, it follows that $\f_\Sigma\in\RPL^+(X)$. If $\Sigma\subset X^\div$, then $D$ is a $\Q$-divisor, and hence $-\p_D\in\PL^+_\hom(X)$. If we further assume $\te(\Sigma)\in\Q$, we get $\f_\Sigma\in\PL^+(X)$, and the remaining implication follows from~\eqref{equ:SigPL}. 
\end{proof}

\begin{exam}\label{exam:absurf2} Suppose $X$ is an abelian surface, $\om=c_1(L)$ with $L\in\Pic(X)_\Q$ ample, and $v=\ord_E$ with $E\subset X$ a prime divisor. Then $\Nef(X)=\Psef(X)$, and $\te(v)=\la_\psef$ is the smallest root of the quadratic equation $(L-\la E)^2=0$. If $X$ has Picard number $\rho(X)\ge 2$, then $\la_\psef$ is irrational for a typical choice of $L$ and $E$, and hence $\f_v\notin\PL(X)$. In particular, $v$ is not dreamy (with respect to $L$) in the sense of Fujita, see Example~\ref{exam:dreamy}. 
\end{exam}
%
\subsection{The Cutkosky example}
Building on a construction of Cutkosky~\cite{Cut} and Proposition~\ref{prop:Zar2D} (itself based on~\cite[\S 6.5]{Per}), we provide an example of a divisorial valuation on $\P^3$ for which~\eqref{equ:fPL1} fails. This relies on the following general result. 

\begin{prop}\label{prop:Cut} Consider a flag of smooth subvarieties $Z\subset S\subset X$ with $\codim S=1$, $\codim Z=2$ and ideals $\fb_S\subset\fb_Z\subset\cO_X$, and assume that 
\begin{itemize}
\item[(i)] $S\equiv \om$; 
\item[(ii)] $\Nef(S)=\Psef(S)$; 
\item[(iii)] $\om|_S-Z$ is not nef on $S$, \ie $\la^S_\nef:=\sup\{\la\ge 0\mid \om|_S-\la [Z]\in\Nef(S)\}<1$. 
\end{itemize}
The Green's function of $v:=\ord_Z\in X^\div$ is then given by 
$$
\f_v=\max\left\{0,\la^S_\nef(\log|\fb_Z|+1),\log|\fb_S|+1\right\}.
$$
In particular, $\te(v)=1$, $\f_v\in\RPL^+(X)$, and 
$$
\f_v\in\PL(X)\Leftrightarrow\f_v\in\PL^+(X)\Leftrightarrow\la^S_\nef\in\Q.
$$
\end{prop}
\begin{proof} Let $\pi\colon Y\to X$ be the blowup along $Z$, with exceptional divisor $E$, and denote by $S'=\pi^\star S-E$ the strict transform of $S$. Since $Z$ has codimension $1$ on $S$, $\pi$ maps $S'$ isomorphically onto $S$, and takes $S'|_{S'}=\pi^\star S|_{S'}-E|_{S'}$ to $S|_S-Z\equiv\om|_S-[Z]$. By (ii) and (iii), we thus have $\Nef(S')=\Psef(S')$, and $S'|_{S'}$ is not nef. 

Consider the cone $C\subset\Num(Y)$ generated by $\theta:=\pi^\star\om\in\Nef(Y)$ and $\a:=-[E]\notin\Psef(Y)$. Since $C$ contains the class of $S'$, it follows from Proposition~\ref{prop:Zar2D} that 
$$
1=\la_\psef:=\sup\{\la\ge 0\mid\pi^\star\om-\la [E]\in\Psef(Y)\}
$$
and $\la\mapsto\Nz(\pi^\star\om-\la E)$ vanishes on $[0,\la^S_\nef]$, and is affine linear on $[\la^S_\nef,1]$, with value $S'$ at $\la=1$. 
By Theorem~\ref{thm:Greenb}, the concave family $(B_\la)_{\la\le 1}$ of $b$-divisors associated to $\f_v$ is affine linear on $(-\infty,0]$, $[0,\la^S_\nef]$ and $[\la^S_\nef,1]$, with value 
$$
B_\la=0,\quad\la^S_\nef\overline E\quad\text{and}\quad\overline{S'+E}=\overline S 
$$ 
at $\la=0$, $\la^S_\nef$ and $1$, respectively. By~\eqref{equ:pshb}, the result follows, since $-\p_{\overline E}=\log|\fb_Z|$ and $-\p_{\overline S}=\log|\fb_S|$. 
\end{proof}


\begin{exam}\label{exam:Cut} Assume $k=\C$, and set $(X,L)=(\P^3,\cO(4))$. By~\cite{Cut}, there exists a smooth quartic surface $S\subset X$ without $(-2)$-curves, and hence such that $\Nef(S)=\Psef(S)$, containing a smooth curve $Z$ such that $\la^S_\nef$ is irrational and less than $1$. By Proposition~\ref{prop:Cut}, we infer $\te(v)=1$ and 
$\f_v\in\RPL^+(X)\setminus\PL(X)$ (in contrast with~\eqref{equ:fPL1}). 
\end{exam}
%
%
\subsection{The Lesieutre example} 
Based on an example by Lesieutre~\cite{Les}, we now exhibit a Green's function that is not $\R$-PL\@. This forms the basis for Theorem~B in the introduction.
\begin{prop}\label{prop:Les} Suppose that $X$ admits a class $\theta\in\Psef(X)$ whose diminished base locus $\B_-(\theta)$ is Zariski dense. Then there exist $\om\in\Amp(X)$ and $v\in X^\div$ such that $Z_X(\f_{\om,v})$ is Zariski dense in $X$. In particular, $\f_{\om,v}\notin\RPL(X)$. 
\end{prop}
\begin{proof} Note first that $\theta$ cannot be big. Otherwise, there would exist an effective $\R$-divisor $D\equiv\theta$, and hence $\B_-(\theta)$ would be contained in $\supp D$. Pick an ample prime divisor $E$ on $X$, choose $c\in\Q_{>0}$ large enough such that $\om:=\theta+c[E]$ is ample, and set $v:=c^{-1}\ord_E\in X^\div$. 
Since $\om$ is ample and $\om-c[E]=\theta$ lies on the boundary of $\Psef(X)$, the threshold $\la_\psef=\sup\{\la\ge 0\mid\om-\la[E]\in\Psef(X)\}$ is equal to $c$. Thus $\B_-(\om-\la_\psef[E])$ is Zariski dense, and hence so is $Z_X(\f_{\om,v})$, by Corollary~\ref{cor:Greenactive}. The last point follows from Lemma~\ref{lem:PLcent}. 
\end{proof}

\begin{exam}\label{exam:Les1} By~\cite[Theorem~1.1]{Les}, the assumptions in Proposition~\ref{prop:Les} are satisfied when $k=\C$ and $X$ is the blowup of $\P^3$ at nine sufficiently general points. 
\end{exam}

If $\theta$ in Proposition~\ref{prop:Les} is rational, then the proof shows that $\om$ can be taken rational as well, \ie $\om=c_1(L)$ for an ample $\Q$-line bundle. While no such rational example appears to be known at present, we can nevertheless exploit the structure of Lesieutre's example to get: 

\begin{prop}\label{prop:Les2} Set $(X,L):=(\P^3,\cO(1))$. Then there exists a finite set $\Sigma\subset X^\div_\R$ such that $Z_X(\f_{L,\Sigma})$ is Zariski dense in $X$, and hence $\f_{L,\Sigma}\notin\RPL(X)$. 
\end{prop}
\begin{proof} Let $\pi\colon Y\to X$ be the blowup at nine sufficiently general points, and denote by $\sum_{i=1}^9 E_i$ the exceptional divisor. By~\cite[Remark~4.5, Lemma 5.2]{Les}, we can pick $D=\sum_i c_i E_i$ with $c_i\in\R_{>0}$ such that the diminished base locus of $\pi^\star L-D$ is Zariski dense. As above, this implies that this class lies on the boundary of the psef cone (it even generates an extremal ray, see~\cite[Lemma~5.1]{Les}), and the psef threshold 
$$
\la_\psef=\sup\{\la\ge 0\mid \pi^\star L-\la D\in\Psef(Y)\}
$$ 
is thus equal to $1$. The result now follows from Corollary~\ref{cor:Greenactive}, with $\Sigma=\{c_i^{-1}\ord_{E_i}\}_{1\le i\le 9}$. 
\end{proof}

It is natural to ask:
\begin{qst} Can an example as in Proposition~\ref{prop:Les2} be found with $\Sigma\subset X^\div$?
\end{qst}
%
%
%
%
\section{The non-trivially valued case}\label{sec:nontriv}
In this section, we work over the non-Archimedean field $K=k\lau{\unipar}$ of formal Laurent series,  with valuation ring $K^\circ:=k\cro{\unipar}$.  We use~\cite{siminag} as our main reference.

Thus $X$ now denotes a smooth projective variety of dimension $n$ over $K$. (In~\S\ref{sec:iso}, it will be obtained as the base change of a smooth projective $k$-variety.)
Working `additively', we view the elements of the analytification $X^\an$ as valuations $w\colon K(Y)^\times\to\R$ for subvarieties $Y\subset X$, restricting to the given valuation on $K$.
%
%
\subsection{Models}\label{sec:models}
We define a \emph{model} of $X$ to be a normal, flat, projective $K^\circ$-scheme $\cX$ together with the data of an isomorphism $\cX_K\simeq X$. The \emph{special fiber} of $\cX$ is the projective $k$-scheme $\cX_0:=\cX\times_{\Spec K}\Spec k$. Each $w\in X^\an$ can be viewed as a semivaluation on $\cX$, whose center is denoted by $\redu_\cX(w)\in\cX_0$. This defines a surjective, continuous \emph{reduction map} $\redu_\cX\colon X^\an\to\cX_0$. For each $w\in X^\an$ we also set 
$$
Z_\cX(w):=\overline{\{\redu_\cX(w)\}}\subset\cX_0. 
$$
The preimage under $\redu_\cX$ of the set of generic points of $\cX_0$ is finite. We denote it by $\Ga_\cX\subset X^\an$, and call its elements the \emph{Shilov points} of $\cX$. As $\cX$ is normal, each irreducible component $E$ of $\cX_0$ defines a \emph{divisorial valuation} $w_E\in X_K^\an$ given by
$$
w_E:=b_E^{-1}\ord_E,\,b_E:=\ord_E(\unipar);
$$
it is the unique preimage under $\redu_\cX$ of the generic point of $E$, and the Shilov points of $\cX$ are exactly these valuations $w_E$.

One says that another model $\cX'$ \emph{dominates} $\cX$ if the canonical birational map $\cX'\dashrightarrow\cX$ extends to a morphism (necessarily unique, by separatedness). In that case, $\redu_\cX$ is the composition of $\redu_{\cX'}$ with the induced projective morphism $\cX'_0\to\cX_0$. The set of models forms a filtered poset with respect to domination. The set
$$
X^\div=\bigcup_\cX\Ga_\cX
$$
of all divisorial valuations is a dense subset of $X^\an$.
%
%
\subsection{Piecewise linear functions}\label{sec:PLnontriv}
A $\Q$-Cartier $\Q$-divisor $D$ on a model $\cX$ of $X$ is \emph{vertical} if it is supported in $\cX_0$; it then defines a continuous function on $X^\an$ called a \emph{model function}.
The $\Q$-vector space $\PL(X)$ of such functions is stable under max, and dense in $\Cz(X^\an)$.  

\begin{defi} We define the space $\RPL(X)$ of \emph{real piecewise linear functions} on $X^\an$ (\emph{$\R$-PL functions} for short) as the smallest $\R$-linear subspace of $\Cz(X^\an)$ that is stable under max (and hence also min) and contains $\PL(X)$.
\end{defi}

Fix a model $\cX$. An ideal $\fa\subset\cO_\cX$ is \emph{vertical} if its zero locus $V(\fa)$ is contained in $\cX_0$. This defines a nonpositive function $\log|\fa|\in\PL(X)$, determined by minus the exceptional divisor of the blowup of $\cX$ along $\fa$, and such that
\begin{equation}\label{equ:logfa}
\log|\fa|(w)<0\Longleftrightarrow Z_\cX(w)\subset V(\fa).
\end{equation}
Functions of the form $\log|\fa|$ for a vertical ideal $\fa\subset\cO_\cX$ span the $\Q$-vector space $\PL(X)$ (see~\cite[Proposition~2.2]{siminag}). As in~\S\ref{sec:PL}, it follows that any function in $\RPL(X)$ can be written as a difference of finite maxima of $\R_+$-linear combinations of functions of the form $\log|\fa|$.  
%
%
\subsection{Dual complexes and retractions}\label{sec:retr}
We use~\cite{MN,siminag} as references. 

 An \emph{snc model} $\cX$ is a regular model $\cX$ such that the Cartier divisor $\cX_0$ has simple normal crossing support. Denote by $\cX_0=\sum_{i\in I} b_i E_i$ its irreducible decomposition. A \emph{stratum} of $\cX_0$ is defined as a non-empty irreducible component of $E_J:=\bigcap_{j\in J} E_j$ for some $J\subset I$. By resolution of singularities, the set of snc models is cofinal in the poset of all models.
  
 The \emph{dual complex} $\D_\cX$ of an snc model $\cX$ is defined as the dual intersection complex of $\cX_0$. Its faces are in 1--1 correspondence with the strata of $\cX_0$, and further come with a natural integral affine structure. In particular, the vertices of $\D_\cX$ are in 1--1 correspondence with the $E_i$'s, and admit a natural realization in $X^\an$ as the set $\Ga_\cX$ of Shilov points $x_{E_i}$. 
 
 This extends to a canonical embedding $\D_\cX\hto X^\an$ onto the set of monomial points with respect to $\sum_i E_i$. The reduction $\redu_\cX(w)\in\cX_0$ of a point $w\in\D_\cX\subset X^\an$ is the generic point of the stratum of $\cX_0$ associated with the unique simplex of $\D_\cX$ containing $x$ in its relative interior. In particular, $Z_\cX(w)$ is a stratum of $\cX_0$.  This embedding is further compatible with the PL structures, in the sense that the $\Q$-vector space $\PL(\D_\cX)$ of piecewise rational affine functions on $\D_\cX$ is precisely the image of $\PL(X)$ under restriction. 
 
 If another snc model $\cX'$ dominates $\cX$, then $\D_\cX$ is contained in $\D_{\cX'}$, and $\PL(\D_{\cX'})$ restricts to $\PL(\D_\cX)$. Furthermore, the set
 \[
   X^{\mathrm{qm}}:=\bigcup_\cX\D_\cX\subset X^\an
 \]
 of \emph{quasimonomial valuations} coincides with the set of Abhyankar points of $X$, see~\cite[Remark~3.8]{siminag} and~\cite[Proposition~3.7]{JM}, while the subset of rational points $\bigcup_\cX\D_\cX(\Q)$ coincides with the set $X^\div$ of divisorial valuations. For later use, we also note:

\begin{lem}\label{lem:centfinite} If $\cX$ is an snc model, then the image $\redu_{\cX'}(\D_\cX)\subset\cX'_0$ of the dual complex of $\cX$ under the reduction map of any other model $\cX'$ is finite.
\end{lem}
\begin{proof} Pick an snc model $\cX''$ that dominates both $\cX$ and $\cX'$. Then $\D_{\cX}$ is contained in $\D_{\cX''}$, and $\redu_{\cX'}(\D_\cX)$ is thus contained in the image of $\redu_{\cX''}(\D_{\cX''})$ under the induced morphism $\cX''_0\to\cX_0$. After replacing both $\cX$ and $\cX'$ with $\cX''$, we may thus assume without loss that $\cX=\cX'$. For any $w\in\D_\cX$, $\redu_\cX(w)$ is then the generic point of some stratum of $\cX_0$, and $\redu_\cX(\D_\cX)$ is thus a finite set. 
\end{proof}

Dually, each snc model $\cX$ comes with a canonical \emph{retraction} $p_\cX\colon X^\an\to\D_\cX$ that takes $w\in X^\an$ to the unique monomial valuation $w'=p_\cX(w)$ such that
\begin{itemize}
\item $Z_\cX(w')$ is the minimal stratum containing $Z_\cX(w)$; 
\item $w$ and $w'$ take the sames values on the $E_i$'s. 
\end{itemize}

This induces a homeomorphism $X^\an\simto\varprojlim_\cX\D_\cX$, which is compatible with the PL structures in the sense that 
\begin{equation}\label{equ:PLmodel}
\PL(X)=\bigcup_\cX p_\cX^\star\PL(\D_\cX). 
\end{equation}
This implies
\begin{equation}\label{equ:RPLmodel}
\RPL(X)=\bigcup_\cX p_\cX^\star\RPL(\D_\cX),
\end{equation} 
where $\RPL(\D_\cX)$ is the space $\R$-PL functions on $\D_\cX$, \ie functions that are real affine linear on a sufficiently fine decomposition of each face into real simplices. 
%
%
\subsection{Psh functions and Monge--Amp\`ere measures}
We use~\cite{siminag,nama,GM} as references. 

A \emph{closed $(1,1)$-form} $\theta\in\cZ^{1,1}(X)$ in the sense of~\cite[\S 4.2]{siminag} is represented by a relative numerical equivalence class on some model $\cX$, called a \emph{determination} of $\theta$. It induces a numerical class $[\theta]\in\Num(X)$. We say that $\theta$ is \emph{semipositive}, written $\theta\ge 0$, if $\theta$ is determined by a nef numerical class on some model. In that case, $[\theta]$ is nef as well. 

To each tuple $\theta_1,\dots,\theta_n$ in $\cZ^{1,1}(X)$ is associated a signed Radon measure $\theta_1\wedge\dots\wedge\theta_n$ on $X^\an$ of total mass $[\theta_1]\inter[\theta_n]$, with finite support in $X^\div$. More precisely, if all $\theta_i$ are determined by a normal model $\cX$, then $\theta_1\wedge\dots\wedge\theta_n$ has support in $\Ga_\cX$ (see~\cite[\S 2.7]{nama}). 

Each $\f\in\PL(X)$ is determined by a vertical $\Q$-Cartier divisor $D$ on some model $\cX$, whose numerical class defines a closed $(1,1)$-form $\ddc\f\in\cZ^{1,1}(X)$. We say that $\f$ is \emph{$\theta$-psh} for a given $\theta\in\cZ^{1,1}(X)$ if $\theta+\ddc\f\ge 0$. 

From now on, we fix a semipositive form $\om\in\cZ^{1,1}(X)$ such that $[\om]$ is ample. A function $\f\colon X^\an\to\R\cup\{-\infty\}$ is \emph{$\om$-plurisubharmonic} (\emph{$\om$-psh} for short) if $\f\not\equiv-\infty$ and $\f$ can be written as the pointwise limit of a decreasing net of $\om$-psh PL functions. The space $\PSH(\om)$ is closed under max and under decreasing limits. 

By Dini's lemma, the space $\CPSH(\om)$ of continuous $\om$-psh functions coincides with the closure in $\Cz(X)$ (with respect to uniform convergence) of the space of $\om$-psh PL functions. 

Each $\f\in\PSH(\om)$ satisfies the `maximum principle' 
\begin{equation}\label{equ:maxpp}
\sup_X\f=\max_{\Ga_\cX}\f
\end{equation}
for any model $\cX$ determining $\om$ (see~\cite[Proposition~4.22]{GM}). For snc models, \cite[\S 7.1]{siminag} more precisely yields: 

\begin{lem}\label{lem:pshretr} Pick $\f\in\PSH(\om)$ and an snc model $\cX$ on which $\om$ is determined. Then: \begin{itemize}
\item[(i)] the restriction of $\f$ to any face of $\D_\cX$ is continuous and convex;
\item[(iii)] the net $(\f\circ p_\cX)_\cX$ is decreasing and converges pointwise to $\f$. 
\end{itemize}
\end{lem}
\begin{rmk}
  The definition of $\PSH(\om)$ given here differs from the one in~\cite{siminag}, but Theorem~8.7 in~\loccit\ implies that the two definitions are equivalent.
\end{rmk}
To each continuous $\om$-psh function $\f$ (or, more generally, any $\om$-psh function of finite energy) is associated its \emph{Monge--Amp\`ere measure} $\MA(\f)=\MA_\om(\f)$, a Radon probability measure on $X$ uniquely determined by the following properties: 
\begin{itemize}
\item $\f\mapsto\MA(\f)$ is continuous along decreasing nets; 
\item if $\f$ is PL, then $\MA(\f)=V^{-1}(\om+\ddc\f)^n$ with $V:=[\om]^n$. 
\end{itemize}
By the main result of~\cite{nama}, any Radon probability measure $\mu$ with support in the dual complex $\D_\cX$ of some snc model can be written as $\mu=\MA(\f)$ for some $\f\in\CPSH(\om)$, unique up to an additive constant. 
%
%
\subsection{Green's functions}\label{sec:green}
As in the trivially valued case, we can consider the Green's function associated to a nonpluripolar set $\Sigma\subset X^\an$. Here we will only consider the following case. Suppose $w\in X^\div$ is a divisorial point, and define
\[
  \f_w:=\f_{\om,w}:=\sup\{\f\in\PSH(\om)\mid \f(w)\le 0\}.
\]
It follows from~\cite[\S8.4]{nama} that $\f_w\in\CPSH(\om)$ satisfies $\MA(\f_w)=\d_w$ and $\f_w(w)=0$.
\begin{prop}\label{prop:greencurve2}
  If $\dim X=1$ and $[\om]$ is a rational class, then $\f_w\in\PL(X)$.
\end{prop}
\begin{proof}
  This follows from Proposition~3.3.7 in~\cite{thuillierthesis}, and can also be deduced from properties of the intersection form on $\cX_0$ for any snc model $\cX$, as in~\cite[Theorem~7.17]{Gub03}.
\end{proof}
This proves part~(i) of Theorem~A in the introduction. We will prove~(ii) in~\S\ref{sec:examiso}.
%
%
\subsection{Invariance under retraction}\label{sec:retrinv}
It will be convenient to introduce the following terminology: 

\begin{defi}\label{defi:retrinv} We say that a function $\f$ on $X^\an$ is \emph{invariant under retraction} if $\f=\f\circ p_\cX$ for some (and hence any sufficiently high) snc model $\cX$ of $X$.
\end{defi}

\begin{exam}\label{exam:retrinvPL} By~\eqref{equ:PLmodel} and~\eqref{equ:RPLmodel}, a function $\f\in\Cz(X^\an)$ lies in $\PL(X)$ (resp.~$\RPL(X)$) iff $\f$ is invariant under retraction and restricts to a $\Q$-PL (resp.~$\R$-PL) function on the dual complex associated to any (equivalently, any sufficiently high) snc model.
\end{exam}

\begin{rmk}\label{rmk:compprop} The condition $\f=\f\circ p_\cX$ in Definition~\ref{defi:retrinv} is stronger than the `comparison property' of~\cite[Definition~3.11]{Li20}, which merely requires $\f=\f\circ p_\cX$ to hold on the preimage under $p_\cX$ of the $n$-dimensional open faces of some dual complex $\D_\cX$, \ie the preimage of the $0$-dimensional strata  of $\cX_0$ under the reduction map. 
\end{rmk}

\begin{prop}\label{prop:invMA} If $\f\in\PSH(\om)$ is invariant under retraction, then $\f\in\CPSH(\om)$, and $\MA(\f)$ is supported in some dual complex. 
\end{prop}
The first point is a direct consequence of Lemma~\ref{lem:pshretr}, while the second one is a special case of the following more precise result. Recall first that the \emph{$\om$-psh envelope} of $f\in\Cz(X^\an)$ is defined as 
$$
\env(f)=\env_\om(f):=\sup\{\f\in\PSH(\om)\mid\f\le f\}. 
$$
By~\cite{siminag}, it lies in $\CPSH(\om)$.

\begin{thm}\label{thm:envdual} For any $\f\in\CPSH(\om)$ and any snc model $\cX$ on which $\om$ is determined, the following properties are equivalent:
\begin{itemize}
\item[(i)] $\MA(\f)$ is supported in $\D_\cX$; 
\item[(ii)] $\f=\env(\f\circ p_\cX)$. 
\end{itemize}
\end{thm}
\begin{proof} For any $\p\in\PSH(\om)$, we have $\p\le\p\circ p_\cX$ (see Lemma~\ref{lem:pshretr}~(iii)), and hence 
\begin{equation}\label{equ:envcX}
\env(\f\circ p_\cX)=\sup\left\{\p\in\PSH(\om)\mid\p\le\f\text{ on }\D_\cX\right\}. 
\end{equation}
Assume (i). By the domination principle (see~\cite[Lemma~8.4]{nama}), any $\p\in\PSH(\om)$ such that $\p\le\f$ on $\supp\MA(\f)\subset\D_\cX$ satisfies $\p\le\f$ on $X$. In view of~\eqref{equ:envcX} this yields (ii). Conversely, assume (ii). For any finite set of rational points $\Sigma\subset\D_\cX(\Q)\subset X^\div$, consider the envelope
$$
\f_\Sigma:=\sup\{\p\in\PSH(\om)\mid\p\le \f\text{ on }\Sigma\}.
$$
Then $\f_\Sigma$ lies in $\CPSH(\om)$, and $\MA(\f_\Sigma)$ is supported in $\Sigma$ (see~\cite[Lemma~8.5]{nama}). The net $(\f_\Sigma)$, indexed by the filtered poset of finite subsets $\Sigma\subset\D_\cX(\Q)$, is clearly decreasing, and bounded below by $\f$. Its limit $\p:=\lim_\Sigma\f_\Sigma$ is thus $\om$-psh, and we claim that it coincides with $\f$. Indeed, we have $\p\le f$ on $\bigcup_\Sigma\Sigma=\D_\cX(\Q)$, and hence on $\D_\cX$, where both $\p$ and $\f$ are continuous. By~\eqref{equ:envcX}, this yields $\p\le\env(\f\circ p_\cX)=\f$. By continuity of the Monge--Amp\`ere operator along decreasing nets, we infer $\MA(\f_\Sigma)\to\MA(\f)$ weakly on $X$, which yields (i) since each $\MA(\f_\Sigma)$ is supported in $\D_\cX$. 
\end{proof}

In view of Proposition~\ref{prop:invMA} and Example~\ref{exam:retrinvPL}, it is natural to conversely ask:  

\begin{qst}\label{qst:inv1} If the Monge--Amp\`ere measure $\MA_\om(\f)$ of $\f\in\CPSH(\om)$ is supported in some dual complex, is $\f$ invariant under retraction? 
\end{qst}
This question appears as~\cite[Question~2]{GJKM}, and is equivalent to asking whether $\f\circ p_\cX$ is $\om$-psh for some high enough model $\cX$, by Theorem~\ref{thm:envdual}. In Example~\ref{exam:Green1} below (see also Theorem~A) we show that the answer is negative. In this example, the support of $\MA_\om(\f)$ is even a finite set. One can nevertheless ask:
\begin{qst}\label{qst:inv2} Assume that $\f\in\CPSH(\om)$ is such that the support of the Monge--Amp\`ere measure $\MA_\om(\f)$ is a finite set contained in some dual complex.
  \begin{itemize}
  \item[(i)]
    is $\f$ $\R$-PL on each dual complex?
  \item[(ii)]
    if $\om$ is rational, is $\f$ $\Q$-PL on each dual complex?
\end{itemize}
\end{qst}
Example~\ref{exam:Green1} below provides a negative answer to (ii). Indeed the function $\f$ in this example is $\R$-PL but not $\Q$-PL, and by~\eqref{equ:PLmodel},~\eqref{equ:RPLmodel}, this implies that $\f$ fails to be $\Q$-PL on some dual complex $\D_\cX$. The answer to (i) is also likely negative in general, as suggested by Nakayama's counterexample to the existence of Zariski decompositions on certain toric bundles over an abelian suface~\cite[IV.2.10]{Nak}.
\begin{qst}\label{qst:inv4}
  Suppose $X$ is a toric variety, and let $\f\in\CPSH(\om)$ be a torus invariant $\om$-psh function such that $\MA_\om(\f)$ supported on a compact subset of $N_\R\subset X^\an$. Is $\f$ invariant under retraction?
\end{qst}
\begin{qst}\label{qst:inv3}
  If $\f\in\CPSH(\om)$ is invariant under retraction, is the same true for $\f|_{Z^\an}$, if $Z\subset X$ is a smooth subvariety?
\end{qst}
%
%
\subsection{The center of a plurisubharmonic function}
We end this section by a version of Theorem~\ref{thm:centpsh1}. In analogy with~\eqref{equ:ZS}, for any subset $S\subset X^\an$ and any model $\cX$ we set 
$$
Z_\cX(S):=\bigcup_{w\in S}Z_\cX(y). 
$$
This is thus the smallest subset of $\cX_0$ that is invariant under specialization and contains the image $\redu_\cX(S)$ of $S$ under the reduction map $\redu_\cX\colon X^\an\to\cX_0$. For any higher model $\cX'$, the induced proper morphism $\cX'_0\to\cX_0$ maps $Z_{\cX'}(S)$ onto $Z_{\cX}(S)$. 

We say that $S\subset X$ is \emph{invariant under retraction} if $p_\cX^{-1}(S)=S$ for some (and hence any sufficiently high) snc model $\cX$. 
\begin{lem}\label{lem:centerpsh} If $S\subset X^\an$ is invariant under retraction, then $Z_\cX(S)$ is Zariski closed for any model $\cX$. 
\end{lem}
\begin{proof} Pick an snc model $\cX'$ dominating $\cX$ such that $S=p_{\cX'}^{-1}(S)$. Since $Z_\cX(S)$ is the image of $Z_{\cX'}(S)$ under the proper morphism $\cX'_0\to\cX_0$, we may replace $\cX$ with $\cX'$ and assume without loss that $\cX=\cX'$. The set $Z_\cX(S)$ obviously contains $Z_\cX(S\cap\D_\cX)$, which is Zariski closed since $Z_\cX(w)$ is a stratum of $\cX_0$ for any $w\in\D_\cX$. Conversely, pick $w\in S$, and set $y:=p_\cX(w)\in\D_\cX$. Then $y\in p_\cX^{-1}(S)=S$, and $Z_\cX(w)\subset Z_\cX(y)$ since it follows from the definition of $p_\cX$ that  $\redu_\cX(w)$ is a specialization of $\redu_\cX(y)$. This shows, as desired, that $Z_\cX(S)=Z_\cX(S\cap\D_\cX)$ is Zariski closed. 
\end{proof}

\begin{defi} Given $\f\in\PSH(\om)$ and a model $\cX$, we define the \emph{center of $\f$ on $\cX$} as 
$$
Z_\cX(\f):=Z_\cX(\{\f<\sup\f\})=\bigcup\{Z_\cX(w)\mid w\in X,\,\f(w)<\sup\f\}.
$$
\end{defi}

\begin{exam} If $\f=\log|\fa|$ for a vertical ideal $\fa\subset\cO_\cX$, then $Z_\cX(\f)=V(\fa)$.
\end{exam}

\begin{thm}\label{thm:centpsh2} For any $\f\in\PSH(\om)$ and any model $\cX$, the following holds: 
\begin{itemize}
\item[(i)] $Z_\cX(\f)$ is an at most countable union of subvarieties of $\cX_0$; 
\item[(ii)] if $\f$ is invariant under retraction, then $Z_\cX(\f)$ is Zariski closed; 
\item[(iii)] $Z_\cX(\f)=\redu_\cX(\{\f<\sup\f\})$; 
\item[(iv)] if $\cX$ determines $\om$, then $Z_\cX(\f)$ is a strict subset of $\cX_0$. 
\end{itemize}
\end{thm}
\begin{qst}
  Is it true that $\{\f<\sup\f\}=\redu_\cX^{-1}(Z_\cX(\f))$ as in Theorem~\ref{thm:centpsh1}?
\end{qst}
\begin{proof} By~\cite[Proposition~4.7]{nama}, $\f$ can be be written as the pointwise limit of a decreasing sequence $(\f_m)_{m\in\N}$ of $\om$-psh PL functions. Since each $\f_m$ is in particular invariant under retraction (see Example~\ref{exam:retrinvPL}), Lemma~\ref{lem:centerpsh} implies that $Z_\cX\{(\f_m<\sup\f\})$ is Zariski closed for each $m$. On the other hand, since $\f_m\searrow\f$ pointwise on $X$, we have $\{\f<\sup\f\}=\bigcup_m\{\f_m<\sup\f\}$, and hence  $Z_\cX(\f)=\bigcup_m Z_\cX(\{\f_m<\sup\f\})$. This proves (i), while (ii) is a direct consequence of Lemma~\ref{lem:centerpsh}. 

 Pick $w\in X^\an$ such that $\f(w)<\sup\f$. To prove (iii), we need to show that any $\xi\in Z_\cX(w)$ lies in $\redu_\cX(\{\f<\sup\f\})$. By Lemma~\ref{lem:pshretr}, we can find a high enough snc model $\cX'$ such that $x':=p_{\cX'}(w)$ satisfies $\f(x')<\sup\f$. By properness of $\cX'_0\to\cX_0$, $Z_{\cX}(w)$ is the image of $Z_{\cX'}(w)$, which is itself contained in $Z_{\cX'}(x')$. After replacing $\cX$ with $\cX'$ and $x$ with $x'$, we may thus assume without loss that $\cX$ is snc and $x$ lies in $\D_\cX$. Pick $y\in X^\an$ with $\redu_\cX(y)=\xi$ (which exists by surjectivity of the reduction map, see~\cite[Lemma~4.12]{Gub13}). Set  $y':=p_\cX(y)$, and denote by $\sigma$ the unique face of $\D_\cX$ that contains $y'$ in its relative interior, the corresponding stratum of $\cX_0$ being the smallest one containing $\xi$. Since the latter lies on the stratum $Z_\cX(w)$, it follows that $\sigma$ contains $x$ (possibly on its boundary). Since $\f$ is convex and continuous on $\sigma$ (see Lemma~\ref{lem:pshretr}), it can only achieve its supremum at the interior point $y'$ if it is constant on $\sigma$. As $w\in\sigma$ satisfies $\f(w)<\sup\f$, it follows that $\f(y')<\sup\f$ as well. Since $y'=p_\cX(y)$, this implies $\f(y)\le\f(y')<\sup\f$ (again by Lemma~\ref{lem:pshretr}). Thus $\xi=\redu_\cX(y)\in\redu_\cX(\{\f<\sup\f\})$, which proves (iii). 

Finally, assume that $\cX$ is normal and determines $\om$. By~\eqref{equ:maxpp}, we can find an irreducible component $E$ of $\cX_0$ whose corresponding Shilov point $w_E\in\Ga_\cX$ satisfies $\f(w_E)=\sup\f$. Since $w_E$ is the only point of $X^\an$ whose reduction on $\cX_0$ is the generic point of $E$, it follows that the latter does not belong to $Z_\cX(\f)$, which is thus a strict subset of $\cX_0$. 
\end{proof}
%
%
\section{The isotrivial case}\label{sec:iso}
We now consider the \emph{isotrivial} case, in which the variety over $K=k\lau{\unipar}$ is the base change $X_K$ of a smooth projective variety $X$ over the (trivially valued) field $k$.
%
%
\subsection{Ground field extension}\label{sec:gfext}
We have a natural projection 
$$
\pi\colon X_K^\an\to X^\an,
$$
while Gauss extension provides a continuous section 
$$
\sigma\colon X^\an\hto X_K^\an
$$
onto the set of $k^\times$-invariant points (see~\cite[Proposition~1.6]{trivvalold}). By~\cite[Corollary~1.5]{trivvalold}, we further have: 
\begin{lem}\label{lem:Gaussdiv} If $v\in X^\an$ is divisorial (resp real divisorial) then  $\sigma(v)\in X_K^\an$ is divisorial (resp.\ quasimonomial).
\end{lem}
The base change of $X$ to $K^\circ$ defines the \emph{trivial model} 
$$
\cX_\triv:=X_{K^\circ}
$$
of $X_K$, whose central fiber will be identified with $X$. More generally, each \emph{test configuration} $\cX\to\A^1$ for $X$ induces via base change $\Spec K^\circ=\Spec k\cro{\unipar}\to\A^1=\Spec k[\unipar]$ a $k^\times$-invariant model of $X_K$, that shares the same vertical ideals and vertical divisors as $\cX$, and will be simply be denoted by $\cX$, for simplicity. 
%
%
\subsection{Psh functions}\label{sec:psh}
For any $\theta\in\Num(X)$, we denote by $\pi^\star\theta\in\cZ^{1,1}(X_K)$ the induced closed $(1,1)$-form, determined by $\theta$ on the trivial model. If $\om\in\Amp(X)$, then $[\pi^\star\om]\in\Num(X_K)$ coincides with the base change of $\om$, and hence is ample. 

\begin{thm}\label{thm:pshpull} Pick $\om\in\Amp(X)$ and $\f\in\PSH(\om)$. Then:
\begin{itemize}
\item[(i)] $\pi^\star\f\in\PSH(\pi^\star\om)$;
\item[(ii)] if $\f$ is further continuous, then $\MA_{\pi^\star\om}(\pi^\star\f)=\sigma_\star\MA_\om(\f)$. 
\end{itemize}
\end{thm}
\begin{lem}\label{lem:PLpull} For any $\f\in\PL(X)$ and $\theta\in\Num(X)$, the following holds:
\begin{itemize}
\item[(i)] $\pi^\star\f\in\PL(X_K)$;
\item[(ii)] $(\pi^\star\theta+\ddc\pi^\star\f)^n=\sigma_\star(\theta+\ddc\f)^n$;
\item[(iii)] $\f$ is $\theta$-psh iff $\pi^\star\f$ is $\pi^\star\theta$-psh. 
\end{itemize}
\end{lem}
\begin{proof} The function $\f$ is determined by a vertical $\Q$-Cartier divisor $D$ on a test configuration $\cX$, that may be taken to dominate the trivial one (see~\cite[Theorem~2.7]{trivval}). The induced vertical divisor on the induced model of $X_K$ then determines $\pi^\star\f$. This proves (i), and also (ii), by comparing~\cite[(2.2)]{nama} and~\cite[(3.6)]{trivval}. Finally, denote by $\theta_\cX$ the pullback of $\theta$ to $\Num(\cX/\A^1)$. Then $\f$ is $\theta$-psh iff $(\theta_\cX+[D])|_{\cX_0}$ is nef, which is also equivalent to $\pi^\star\f$ being $\pi^\star\theta$-psh. This proves (iii). 
\end{proof}
\begin{proof}[Proof of Theorem~\ref{thm:pshpull}] Write $\f$ as the limit on $X^\an$ of a decreasing net of $\om$-psh PL functions $\f_i$. By Lemma~\ref{lem:PLpull}, $\pi^\star\f_i$ is PL and $\pi^\star\om$-psh. Since it decreases pointwise on $X_K^\an$ to $\pi^\star\f$, the latter is $\pi^\star\om$-psh, which proves (i). For each $i$, Lemma~\ref{lem:PLpull}~(ii) further implies $\MA_{\pi^\star\om}(\pi^\star\f_i)=\sigma_\star\MA_\om(\f_i)$. If $\f$ is continuous, then $\MA_\om(\f)$ and $\MA_{\pi^\star\om}(\pi^\star\f)$ are both defined, and are the limits of $\MA_\om(\f_i)$ and $\MA_{\pi^\star\om}(\pi^\star\f_i)$, respectively. This proves (ii). 
\end{proof}
%
%
\subsection{PL structures}\label{sec:gfPL}
As a direct consequence of Lemma~\ref{lem:PLpull}, the projection $\pi\colon X_K^\an\to X^\an$ is compatible with the PL structures:

\begin{cor}\label{cor:PLpull} We have $\pi^\star\PL(X)\subset\PL(X_K)$ and $\pi^\star\RPL(X)\subset\RPL(X_K)$. 
\end{cor}

As we next show, this is also the case for Gauss extension. 
\begin{thm}\label{thm:PLGauss} We have $\sigma^\star\PL(X_K)=\PL(X)$   and $\sigma^\star\RPL(X_K)=\RPL(X)$.
\end{thm}

Any vertical ideal $\fa$ on $\cX_\triv$, being trivial outside the central fiber, can be viewed as a vertical ideal on $X\times\A^1$, and $\widetilde\fa:=\Gm\cdot\fa$ is then the smallest flag ideal containing $\fa$. 

\begin{lem}\label{lem:PLGauss} With the above notation we have $\sigma^\star\log|\fa|=\f_{\widetilde\fa}$. 
\end{lem}
\begin{proof} Pick an ample line bundle $L$ on $X$, and denote by $\cL_\triv$ the trivial model of $L_K$, \ie the pullback of $L$ to the trivial model $\cX_\triv=X_{K^\circ}$. After replacing $L$ with a large enough multiple, we may assume $\cL_\triv\otimes\fa$ is generated by finitely many sections $s_i\in\Hnot(\cX_\triv,\cL_\triv)$. Then $\log|\fa|=\max_i\log|s_i|$, where $|s_i|$ denotes the pointwise length of $s_i$ in the model metric induced by $\cL_\triv$. For each $i$ write 
$s_i=\sum_{\la\in\Z} s_{i,\la} \unipar^{\la}$ where $s_{i,\la}\in\Hnot(X,L)$, and denote by $\fb_\la\subset\cO_X$ the ideal locally generated by $(s_{i,\la})_i$. Then $\widetilde\fa=\sum_{\la\in\Z}\fb_\la\unipar^\la$. By definition of Gauss extension, we have for any $v\in X^\an$ 
$$
\log|s_i|(\sigma(v))=\max_{\la\in\Z}\{\log|s_{i,\la}|+\la\}. 
$$
Thus $\sigma^\star\log|\fa|=\max_{\la\in\Z}\{\p_\la-\la\}$ with $\p_\la:=\max_i\log|s_{i,\la}|=\log|\fb_\la|$, and hence $\sigma^\star\log|\fa|=\max_\la\{\log|\fb_\la|-\la\}=\f_{\widetilde\fa}$. 
\end{proof}

\begin{proof}[Proof of Theorem~\ref{thm:PLGauss}]
  By Corollary~\ref{cor:PLpull} we have $\pi^\star\PL(X)\subset\PL(X_K)$.
  Since $\PL(X_K)$ is generated by functions of the form $\log|\fa|$ for a vertical ideal $\fa\subset\cO_{\cX_\triv}$, Lemma~\ref{lem:PLGauss} yields $\sigma^\star\PL(X_K)\subset\PL(X)$, and hence also $\sigma^\star\RPL(X_K)\subset\RPL(X)$. This completes the proof, since $\sigma^\star\pi^\star=\id$.
\end{proof}
%
%
\subsection{Centers}\label{sec:gfcent}
Next we study the relationships between the two center maps $Z_X\colon X^\an\to X$ and $Z_{\cX_\triv}\colon X_K^\an\to\cX_{\triv,0}=X$.
\begin{lem}\label{lem:centers} For all $w\in X_K^\an$ and $v\in X^\an$ we have 
$$
Z_{\cX_\triv}(w)\subset Z_X(\pi(w)),\quad Z_X(v)=Z_{\cX_\triv}(\sigma(v)). 
$$ 
\end{lem}

\begin{proof} Denote by $\fb\subset\cO_X$ the ideal of the subvariety $Z_X(\pi(w))$. Then $\fa:=\fb+(\unipar)$ is a vertical ideal on $\cX_\triv$ such that $V(\fa)=V(\fb)=Z_X(\pi(w))$ under the identification $\cX_{\triv,0}=X$. Further,
$$
\log|\fa|(w)=\max\{\log|\fb|(\pi(w)),-1\}<0,
$$
and hence $Z_{\cX_\triv}(w)\subset V(\fa)=Z_X(\pi(w))$,  see~\eqref{equ:logfa}. 

In particular, $Z_{\cX_\triv}(\sigma(v))\subset Z_X(v)$.
Conversely, denote by $\fa\subset\cO_{\cX_\triv}$ the ideal of $Z_{\cX_\triv}(\sigma(v))$. Since $\sigma(v)$ is $k^\times$-invariant, $\fa=\sum_{\la\in\Z}\fa_\la\unipar^{-\la}$ is (induced by) a flag ideal.  Further, $\f_\fa(v)=\log|\fa|(\sigma(v))<0$, and hence $Z_X(v)\subset Z_X(\f_\fa)$. By Example~\ref{exam:PLhom} we have $Z_X(\f_\fa)=V(\fa_0)$. The latter is also equal to the zero locus of $\fa_0+(\unipar)$ on $\cX_\triv$, which is contained in $V(\fa)=Z_{\cX_\triv}(\sigma(v))$ since $\fa\subset\fa_0+(\unipar)$. Thus $Z_X(v)\subset Z_{\cX_\triv}(\sigma(v))$, which concludes the proof.
\end{proof}
\begin{prop}\label{prop:gfcent} If $\om\in\Amp(X)$ and $\f\in\PSH(\om)$, then
$Z_{\cX_\triv}(\pi^\star\f)=Z_X(\f)$.
\end{prop}
\begin{proof}
  Pick $v\in X^\an$ such that $\f(v)<\sup\f$, and set $w:=\sigma(v)$. Then  $\pi^\star\f(w)=\f(v)$ and $\sup\pi^\star\f=\sup\f$, so $w$ lies in $\{\pi^\star\f<\sup\pi^\star\f\}$, and hence $Z_X(v)=Z_{\cX_\triv}(w)\subset Z_{\cX_\triv}(\pi^\star\f)$ by Lemma~\ref{lem:centers}. This implies $Z_X(\f)\subset Z_{\cX_\triv}(\pi^\star\f)$. Conversely, assume $w\in X_K^\an$ satisfies $\pi^\star\f(w)<\sup\pi^\star\f$. Then $v:=\pi(w)$ lies in $\{\f<\sup\f\}$, and hence $Z_X(v)\subset Z_X(\f)$. In view of Lemma~\ref{lem:centers}, this implies  $Z_{\cX_\triv}(w)\subset Z_X(\f)$, and hence $Z_{\cX_\triv}(\pi^\star\f)\subset Z_X(\f)$. \end{proof}
Combining Proposition~\ref{prop:gfcent} and Theorem~\ref{thm:centpsh2}, we obtain
\begin{cor}\label{cor:critinvretr}
  Let $\f\in\PSH(\om)$, where $\om\in\Amp(X)$, and suppose that $\pi^\star\f\in\PSH(\pi^\star\om)$ is invariant under retraction. Then $Z_X(\f)\subset X$ is a Zariski closed proper subset of $X$. \end{cor}
%
%
\subsection{Examples}\label{sec:examiso}
We are now ready to prove Theorems~A and~B in the introduction, and also provide additional examples. As in the previous section, $X$ denotes a smooth projective variety over $k$. Pick a class $\om\in\Amp(X)$, a $k^\times$-invariant divisorial point $w\in X_K^\div$, and denote as in~\S\ref{sec:green} by $\f_w\in\CPSH(\pi^\star\om)$ the Green's function associated to $w$; this is the unique solution to the Monge--Amp\`ere equation 
\[
  \MA_{\pi^\star\om}(\f_w)=\d_w\quad\text{and}\quad
  \f_w(w)=0.
\]
By Lemma~\ref{lem:Gaussdiv}, we have $w=\sigma(v)$ with $v:=\pi(w)\in X^\div$. If  $\f_v\in\CPSH(\om)$ denotes the Green's function of $\{v\}$, see~\S\ref{sec:equmeas}, then we have
\[
  \f_w=\pi^\star\f_v.
\]
Indeed,  $\pi^\star\f_v(w)=\f_v(v)=0$, and by Theorem~\ref{thm:pshpull}, we have $\MA_{\pi^\star\om}(\pi^\star\f_v)=\sigma_\star\d_v=\d_w$.

Our goal is to investigate the regularity of $\f_w$. 
\begin{cor}\label{cor:Greenglob}
  If $\dim X=1$, then $\f_w\in\PL(X_K)$. If $\dim X=2$, then $\f_w\in\RPL(X_K)$.
\end{cor}
\begin{proof}
  The first statement follows from Proposition~\ref{prop:greencurve2}. Now suppose $\dim X=2$. 
  By Theorem~\ref{thm:Greensurf}, $\f_v\in\RPL(X)$, so that $\f_w\in\RPL(X_K)$, see Corollary~\ref{cor:PLpull}.
\end{proof}
However, even when $\om$ is rational, $\f_w$ is in general not $\Q$-PL:
\begin{exam}\label{exam:Green1}
  Example~\ref{exam:absurf2} gives an example of an abelian surface $X$, a rational class $\om\in\Amp(X)$, and a divisorial valuation $v\in X^\div$ such that $\f_v\in\RPL(X)\setminus\PL(X)$. If $w=\sigma(v)$, then $\f_w:=\pi^\star\f_v\in\RPL(X_K)\setminus\PL(X_K)$, by Theorem~\ref{thm:PLGauss}. 
\end{exam}
\begin{exam}\label{exam:Green3}
  Similarly, Example~\ref{exam:Cut} gives an example of a divisorial valuation $v\in \P^{3,\div}$ such that if we set $\om=c_1(\cO(4))$, then $\f_v:=\f_{\om,v}\in\RPL(X)\setminus\PL(X)$. If $w=\sigma(v)$, then $\f_w:=\pi^\star\f_v\in\RPL(X_K)\setminus\PL(X_K)$, by Theorem~\ref{thm:PLGauss}. 
\end{exam}
Examples~\ref{exam:Green1} and~\ref{exam:Green3} establish Theorem~A~(ii). They  also provide a negative answer to Question~\ref{qst:inv2}~(ii). Indeed, a function $\f\in\Cz(X_K^\an)$ lies in $\RPL(X_K)$ (resp.~$\PL(X_K)$) iff $\f$ is invariant under retraction and restricts to an $\R$-PL (resp.~$\Q$-PL) function on each dual complex, see Example~\ref{exam:retrinvPL}.

\smallskip
As the next example shows, if $\dim X=3$, then $\f_w$ need not be $\R$-PL\@. In fact, it may not even be invariant under retraction.
\begin{exam}\label{exam:Green2}
  Example~\ref{exam:Les1} shows that we may have $\dim X=3$ and $Z_X(\f_v)$ Zariski dense in $X$. It follows that $Z_{\cX_{\triv}}(\f_w)$ is Zariski dense in $\cX_{\triv,0}=X$, see
Theorem~\ref{thm:pshpull}~(ii). Thus Theorem~\ref{thm:centpsh2}~(ii) shows that $\f_w$ cannot be invariant under retraction.  
\end{exam}
It could, however, a priori be the case that the restriction $\f_w$ to any dual complex is $\R$-PL, see Question~\ref{qst:inv2}~(i). 

In Example~\ref{exam:Green2}, based on Lesieutre's work, the class $\om$ is irrational. We do not know of an example for which the class $\om$ is rational. However, the following example provides a proof of Theorem~B in the introduction.
\begin{exam}\label{exam:Lesglob}
  Set $X=\P^3_k$ and $\om:=c_1(\cO(1))\in\Num(X)$.
By Proposition~\ref{prop:Les2}, there exists $\p\in\CPSH(\om)$ such that $\MA_\om(\p)$ is supported in a finite subset $\Sigma\subset X^\div_\R$, and $Z_X(\p)$ is Zariski dense in $X$.
Theorem~\ref{thm:pshpull} then shows that $\f:=\pi^\star\p$ lies in $\CPSH(\pi^\star\om)$, $\MA_{\pi^\star\om}(\f)=\sigma_\star\MA_\om(\p)$ has finite support in some dual complex (see Lemma~\ref{lem:Gaussdiv}), and the center of $\f$ on the trivial model of $X_K^\an$ is Zariski dense. By Theorem~\ref{thm:centpsh2}, it follows that $\f$ cannot be invariant under retraction. 
\end{exam}
%
%
%
%

%
%
%
%
%
%
\end{document}